\documentclass[12pt]{extarticle}
\usepackage{amsmath, amsthm, amssymb, hyperref, color, bm}
\usepackage[shortlabels]{enumitem}
\usepackage{graphicx}
\usepackage[all]{xypic}
\usepackage{makecell}
\usepackage[final]{pdfpages}
\setboolean{@twoside}{false}
\usepackage{pdfpages}
\usepackage{caption}
\usepackage{subcaption}
\usepackage{scalefnt}
\usepackage{verbatim}
\oddsidemargin \evensidemargin
\newtheorem{theorem}{Theorem}
\numberwithin{theorem}{section}
\newtheorem{proposition}[theorem]{Proposition}
\newtheorem{lemma}[theorem]{Lemma}
\newtheorem{corollary}[theorem]{Corollary}

\newtheorem{remark}[theorem]{Remark}
\newtheorem{example}[theorem]{Example}

\newcommand{\RP}{\mathbb{RP}}
\newcommand{\RR}{\mathbb{R}}
\newcommand{\QQ}{\mathbb{Q}}
\newcommand{\PP}{\mathbb{P}}
\newcommand{\CC}{\mathbb{C}}
\newcommand{\ZZ}{\mathbb{Z}}
\def\qmatrix#1{\left[\begin{matrix}#1\end{matrix}\right]} 
 
\newcommand{ \rrk }{{\mathrm{rk}_{\mathbb{R}}}}
\newcommand{ \crk }{{\mathrm{rk}_{\mathbb{C}}}}
 \date{}
\newcommand{\SSS}{S}
\newcommand{\BQ}{\Omega}

\newcommand{\note}[2]{{\marginpar{\tiny{ \bf #1: }{\color{red}#2}}}}
 
\title{\textbf{Congruences and Concurrent Lines\\ in Multi-View Geometry}}

\author{Jean Ponce, Bernd Sturmfels and Matthew Trager} 

\begin{document}

\maketitle

\begin{abstract} \noindent We present a new framework for multi-view
  geometry in computer vision. A camera is a mapping
  between $\PP^3$ and a line congruence. This model, which ignores
  image planes and measurements, is a natural abstraction of
  traditional pinhole cameras. It includes two-slit cameras, pushbroom cameras,
  catadioptric cameras, and many more.  We study the concurrent lines
  variety, which consists of $n$-tuples of lines in $\PP^3$ that
  intersect at a point. Combining its equations with those of various
  congruences, we derive constraints for
  corresponding images in multiple views. 
  We also study photographic cameras which use image measurements and
  are modeled as rational maps from $\PP^3$ to $\PP^2$ or $\PP^1\times\PP^1$.
  \end{abstract}

\section{Introduction}

Multi-view geometry lays the
foundations for algorithms that reconstruct a scene from multiple
images. Developed in the 1980's, building on classical
photogrammetry, this subject has had many successful
applications in computer vision. The book \cite{HZ} offers a comprehensive
introduction. Recently, on the mathematical side,
the field of {\em algebraic vision} emerged.
It studies objects such as the multi-view varieties \cite{AST, THP} and their
moduli in tensor spaces~\cite{AO, Oed}.

A pinhole camera is typically modeled as a linear map $\PP^3
\dashrightarrow \PP^2$, described by a $3 \times 4$-matrix up to scale.
This has eleven degrees of freedom, three of which describe the center
(or pinhole) in $\PP^3$, while the remaining eight degrees of freedom account for the choice
of image coordinates. In this paper we distinguish between
traditional {\em photographic} cameras that use image measurements,
and {\em geometric} ones that do not require fixing coordinate
systems, but map points onto the corresponding viewing rays. 
We work with a
generalized notion of camera, where the family of these rays is not
necessarily focused at a pinhole. This includes several practical
devices, such as pushbroom, panoramic and catadioptric
cameras~\cite{SRT}. 

The main requirement for any camera model is 
that the fibers of all image points must be lines. This is essential since light travels
along lines. With this condition, a {\em photographic camera} is for us a map
$\PP^3 \dashrightarrow \PP^2$ or $\PP^3 \dashrightarrow \PP^1 \times
\PP^1$, where $\PP^2$ or $\PP^1 \times \PP^1$ is the space of image
measurements. A {\em geometric camera} is instead a map $\PP^3 \dashrightarrow {\rm
  Gr}(1,\PP^3)$ from $3$-space into the Grassmannian of lines. 
The latter is an abstraction of a physical camera, which ignores part of the 
image formation process, namely the mapping from viewing rays 
to coordinates. In this paper, we focus 
mostly on this type of geometric cameras.
We will also assume that the coordinates of the map from points to
lines are algebraic functions.  A geometric camera is always associated with a
{\em congruence} of lines \cite{Jes}, i.e.,~a two-dimensional family of lines, that is the 
image of the camera in the Grassmannian ${\rm Gr}(1,\PP^3)$.  Indeed, it 
has already been argued that congruences should play a central role in multi-view geometry,
e.g., \cite{BGP,Paj,Pon}. In this setting, congruences of order one
\cite{Kum} are of particular interest. These define {\em rational}
geometric cameras, where the map from points to image lines is given
by rational functions. For example, a
pinhole camera is associated with the bundle of lines passing through a fixed point
in $\PP^3$, and the action of camera takes a point in $\PP^3$ to the line
joining it to the pinhole. 
A {\em two-slit camera} is associated with the common transversals
of two lines $\ell_1$ and $\ell_2$ in $\PP^3$ (the slits), and taking the picture of a world point 
$x$ now means mapping $x$ to the  line through $x$ that intersects both
$\ell_1$ and $\ell_2$.  Other rational cameras arise from the common
transversals to an algebraic space curve $C$ of degree $d$ and a line
$\ell$ meeting $C$ in $d-1$ points.

Taking pictures with $n$ rational cameras for congruences $C_1,\ldots,C_n$ defines
a rational~map
\begin{equation}
\label{eq:phimap}
 \phi \,:\,\PP^3 \,\dashrightarrow \,C_1 \times C_2 \times \cdots \times C_n
\,\subset\, ({\rm Gr}(1,\PP^3))^n \,\subset \, (\PP^5)^n. 
\end{equation}
The rightmost inclusion is
the Pl\"ucker embedding of the Grassmannian.
The surface $C_i$ now plays the role of the $i$-th image plane $\PP^2$ in
classical multi-view geometry \cite{AST,HZ}. Our main object of study in this paper is the image
of the map $\phi$. This lives in $ ({\rm Gr}(1,\PP^3))^n$ and hence in $(\PP^5)^n$.
The Zariski closure of this image is an irreducible projective variety of dimension~$3$. 
We call this variety the {\em multi-image variety} of the $n$-tuple of 
cameras $(C_1,C_2,\ldots,C_n)$. To characterize it,
 we study the variety $V_n$ of $n$-tuples
of concurrent lines in $\PP^3$. Under
suitable genericity assumptions, the multi-image variety 
equals the intersection
\begin{equation}
\label{eq:multiview}
 (C_1 \times C_2 \times \cdots \times C_n) \,\cap \,V_n 
 \qquad {\rm in} \quad {\rm Gr}(1,\PP^3)^n \,\subset\,(\PP^5)^n.
 \end{equation}

We next discuss the organization of the paper and summarize our main contributions.
 In  Section \ref{sec2} we fix our notation for
 Pl\"ucker coordinates of lines, and we review the geometry of congruences.
We show how to compute the focal locus of a congruence, 
and  we discuss how a congruence can be recovered from its focal locus.
In Section~\ref{sec3} we study the $(2n+3)$-dimensional variety $V_n$ of $n$-tuples
of concurrent lines in $\PP^3$. 
Our main result  (Theorem~\ref{thm:conclines}) characterizes  minimal ideal generators
and a Gr\"obner basis for $V_n$.
In Section~\ref{sec4} we study congruences of order one.
These were classified in 1866 by Kummer \cite{Kum}.
We revisit his classification from a computer vision perspective, 
and we derive formulas for the associated rational cameras.
Section~\ref{sec5} introduces the multi-image variety
for $n$ rational cameras. If each $C_i$ is a pinhole camera then this is isomorphic to the
 familiar multi-view variety \cite{AST}.  In Section~\ref{sec6} we
 study geometric cameras of order greater than one.  Here the
 point-to-line map is algebraic but not rational.  These include
 panoramic and catadioptric cameras.  Section~\ref{sec7} contains a
 brief discussion on photographic cameras. We point out the
 relationship between general multi-view constraints and the
 concurrent lines ideal. As concrete application, we extend the familiar
 fundamental matrix to the context of linear two-slit projections.

Our presentation is intended for a diverse audience, ranging from students in
mathematics to researchers in computer vision. The prerequisites in
algebraic geometry are minimal. We shall assume familiarity with
ideals and varieties at the level of the
undergraduate text \cite{CLO}.

\section{Lines and Congruences}
\label{sec2}

This section collects basics on the Grassmannian of lines in $3$-space and congruences of lines, that will be central for our discussion. We work in projective spaces $\PP^n$ over the field $\CC$ of complex numbers. 
Our varieties will be defined by polynomials
 that have coefficients in the field $\RR$ of real numbers, and we will be mostly interested 
 in the real locus of these varieties.

\subsection{The Grassmannian of Lines}

The {\em Grassmannian} ${\rm Gr}(1,\PP^3)$ of lines in $\PP^3$ is a $4$-dimensional 
manifold. 
The line through points $x = (x_0:x_1:x_2:x_3)$ and $y = (y_0:y_1:y_2:y_3)$ in $\PP^3$   
has {\em Pl\"ucker coordinates} $p_{ij}=x_iy_j-x_jy_i$.
The point $(p_{01}:p_{02}:p_{03}:p_{12}:p_{13}:p_{23})$ in  $\PP^5$
is independent of the choice of $x$ and $y$, and satisfies 
$p_{03}p_{12} - p_{02}p_{13} + p_{01} p_{23}=0$. All solutions to this equation come from a line, so we identify ${\rm Gr}(1,\PP^3)$ with  the {\em Pl\"ucker quadric}
$V(p_{03}p_{12} - p_{02}p_{13} + p_{01} p_{23})$ in~$\PP^5$.

We can also represent a line as the intersection of two planes.
Each plane $\{  u^0 x_0 + u^1 x_1 + u^2 x_2 + u^3 x_3 = 0\}$ in $\PP^3$
corresponds to a point $ (u^0:u^1:u^2:u^3)$ in the dual projective space $(\PP^3)^*$.
The line that is the  intersection of 
the planes $(a^0 : a^1 : a^2 : a^3)$ and $(b^0 : b^1 : b^2 : b^3)$ has {\em dual Pl\"ucker coordinates} $(p^{01}:p^{02}:p^{03}:p^{12}:p^{13}:p^{23})$ where $p^{ij}=a^ib^j-a^jb^i$. Primal and dual coordinates are related via $p^{ij}=\sigma_{(ijkl)} p_{kl}$, where $i,j,k,l$ are distinct indices and $\sigma_{(ijkl)}$ denotes the sign of the permutation $(ijkl)$. Alternatively, the duality between points and planes in $\PP^3$ given by the usual dot product induces an involution on the Pl\"ucker quadric that maps a line $p=(p_{ij})$ to a 
     dual line $p^*=(p_{23}:-p_{13}:p_{12}:p_{03}:-p_{02}:p_{01})$.

To express incidences of lines with points and planes,
it is convenient to write the Pl\"ucker coordinates
  of a line $p$ and its dual $p^*$ as the entries of two skew-symmetric $4 \times 4$-matrices:
\begin{equation}
\label{eq:PPprime}
 P = \begin{bmatrix}
   0 &  p_{23} & \!\! -p_{13} &   p_{12} \\
-p_{23} &  0  &  p_{03} & \!\! -p_{02} \\
 p_{13} &\!\! -p_{03} &   0 &  p_{01} \\
-p_{12} & p_{02} &  \!\! -p_{01} &    0  
\end{bmatrix} \qquad {\rm and} \quad\,\,\,
 P^* = \begin{bmatrix}
   0 &   p_{01} &  p_{02} & p_{03} \\
-p_{01} &   0 &  p_{12} & p_{13} \\
-p_{02} &\!\! -p_{12} &  0  &  p_{23} \\
-p_{03} & \!\! -p_{13} & \!\!\! -p_{23} &  0 
\end{bmatrix}.
\end{equation}
If $x$ and $y$ are column vectors representing points
on the line $p$, then our definition for the associated matrix $P^*$ is simply
$x y^T - y x^T$. The conditions $\,{\rm rank}(P) = 2$, $\,{\rm
  rank}(P^*) = 2$, and ${\rm trace}(P P^*) = 0$ are all equivalent to
the Pl\"ucker quadric that cuts out ${\rm Gr}(1,\PP^3)$ inside $\PP^5$.

Concurrent lines are characterized as follows:  if $q$ is an additional line represented 
by matrices $Q$ and $Q^*$ as above, then the lines 
$p$ and $q$ intersect in $\PP^3$ if and only if the bilinear form
${\rm trace}(P Q^*) \,= \,{\rm trace}(P^*Q)\,$ vanishes.
In particular, all lines that intersect a fixed line $p$ form a threefold in $\PP^5$, obtained by intersecting ${\rm Gr}(1,\PP^3)$ with a hyperplane.

Given a point $x$ in $\PP^3$, the line $p$ contains $x$ if and only if $P x=0$. 
This yields three independent linear equations in the entries of $P$. They define
a plane in $\PP^5$ contained in the Grassmannian ${\rm Gr}(1,\PP^3)$, known as the
 {\em $\alpha$-plane} of $x$. Similarly, if $u$ is a plane in $\PP^3$, then 
$u$ contains the line $p$ if and only if $P^* u = 0$. This defines a plane
   in  ${\rm Gr}(1,\PP^3)$, namely the {\em $\beta$-plane} of $u$. 
  The families of $\alpha$ and $\beta$-planes form two disjoint rulings on the Pl\"ucker quadric. Two different planes in the same family ($\alpha$ or $\beta$) always intersect in exactly one point in ${\rm Gr}(1,\PP^3)$.
  On the other hand, the $\alpha$-plane of $x$ and the $\beta$-plane of $u$ do not meet unless $x$ lies on $u$.
  Throughout this paper, we use the standard notation for  {\em join}
  ($\vee$) and {\em meet}  ($\wedge$) of linear spaces. For example,
  given $x,y$ in $\PP^3$, we write $x \vee y$ for the line they span.

Finally, if $p$ is a line and $x$ is a point not in $p$, then the non-zero vector $Px$ 
represents the plane that contains both $p$ and $x$. On the dual side,
if $u$ is a plane not containing the line $p$, then the non-zero vector $P^* u$
represents the intersection point of $u$ and $p$.

\subsection{Congruences}

A surface $C$ in ${\rm Gr}(1,\PP^3)$ represents a two-dimensional
family of lines in $\PP^3$. This is classically known as a {\em congruence} \cite{DeP, DM}.
The {\em bidegree} $(\alpha, \beta)$ of a congruence $C$ is a pair of nonnegative integers that
represents the  class of $C$ in the cohomology of ${\rm Gr}(1,\PP^3)$. 
The {\em order}  $\alpha$ is the number of lines in $C$ that pass through
a general point of $\PP^3$, while the {\em class} $\beta$ is the number
of lines in $C$ that lie in a general plane of $\PP^3$.
The study of congruences was an active area of research in the second half of the $19$-th century.
Many results from that period can be found in the book by Jessop \cite{Jes} on {\em line complexes}, 
 the classical term for threefolds in ${\rm Gr}(1,\PP^3)$. 

\begin{example} \rm {\em (1,0) and (0,1)-Congruences.} A congruence $C$ has bidegree $(1,0)$ if and only if it is an $\alpha$-plane for some point $x$ in $\PP^3$ ($C$ is the set of lines through $x$). Dually, 
a congruence $C$ has bidegree $(0,1)$ if and only if it is a $\beta$-plane for some plane $u$ in $\PP^3$. \hfill $\diamondsuit$
\end{example}

 Given an $(\alpha,\beta)$-congruence $C$, a point $x \in \PP^3$ is a {\em focal point}
 if $x$ does not belong to $\alpha$ distinct lines of $C$.
  This may happen if $x$ belongs to fewer than $\alpha$ distinct 
 lines, or if $x$ belongs to an infinite number of lines. In the latter case,
 $x$ is a {\em fundamental point}. The variety $\mathcal{F}(C)$ of focal points is the {\em focal locus},
  while the variety $\mathcal G(C)$ of fundamental points is the {\em fundamental locus}. 
 Clearly, $\mathcal G(C)$ is contained in $\mathcal F(C)$.
 Moreover, the focal locus $\mathcal F(C)$ is typically a surface in $\PP^3$. 
 It is known (cf.~\cite[Proposition 2]{DM}) that $\mathcal{F}(C)$ has lower dimension if and only
  if $C$ has order at most one, in which case $\mathcal F(C)=\mathcal G(C)$.
The image of $C$ under the map $p \mapsto p^*$
is denoted $C^*$. This {\em dual congruence} has bidegree $(\beta,\alpha)$.
 The focal locus $\mathcal{F}(C^*)$  of the dual congruence is
the projectively dual variety of the focal locus~$\mathcal{F}(C)$.

Two natural congruences are derived
from geometric objects in $\PP^3$.
Given a surface $X$ in $\PP^3$, we consider the set of all lines
that are tangent to $X$ at two points.
These bitangents satisfy two constraints, so they form a surface $\mathcal{B}(X)$
in ${\rm Gr}(1,\PP^3)$. We call this the {\em bitangent congruence} of $X$.
For a curve $Y$ in $\PP^3$, we consider the set of lines
that intersect $Y$ in two points. These lines form the {\em secant congruence} $\mathcal{S}(Y)$.
The following classical result (cf~\cite[\S 281]{Jes}) can be regarded as the fundamental
theorem on congruences. See also \cite{ABT, BG, DeP, Kum}.

\begin{theorem}
Let $C \subset {\rm Gr}(1,\PP^3)$ be an irreducible congruence. If the focal locus of $C$ is a surface
$X$ then $C$ is an irreducible component of the bitangent congruence $\mathcal{B}(X)$.
 If the focal locus of $C$ is a curve
$Y$ then $C$ is an irreducible component of the secant congruence $\mathcal{S}(Y)$.
\end{theorem}

\begin{example} \label{ex:twotwo} \rm Consider a congruence $C$ that
  is defined by a general linear form and a general quadratic form in
  Pl\"ucker coordinates. Then $C$ has bidegree $(2,2)$. We can view
  $C$ as the intersection of two general quadrics in $\PP^4$, so it is
  a {\em del Pezzo surface of degree four}.  There are $16$ straight
  lines in $C$.  Each of these is a pencil of lines in $\PP^3$ that
  pass through a point and belong to a plane.  These $16$ points in
  $\PP^3$ form the fundamental locus $\mathcal{G}(C)$, and
  the $16$ planes form the fundamental locus of the dual congruence
  $C^*$. The focal locus  $X=\mathcal F(C)$ is a {\em Kummer
    surface}, that is, a quartic surface in $\PP^3$ with exactly $16$
  nodes. The bitangent congruence for $X$ contains $C$ but also five
  other similar $(2,2)$-congruences, and sixteen $(0,1)$-congruences,
  associated with the fundamental planes. See \cite[Example 5.5]{ABT}.
  \hfill $\diamondsuit$
\end{example}

The focal locus $\mathcal{F}(C)$ of a congruence $C$
can be computed as follows. Let $I$ be the ideal 
in $\CC[p_{01},p_{02},p_{03},p_{12},p_{13},p_{23}]$  that defines $C$.
Of course, $p_{03}p_{12} - p_{02} p_{13} + p_{01} p_{23} \in I$.
The set of lines in $C$ that pass through a point $x=(x_0 : x_1 : x_2 : x_3)$ in $\PP^3$ is
given by the ideal
\begin{equation}\label{eq:system}
 I  \,+ \,\langle P x  \rangle ,
\end{equation}
where $P$ is the $4 \times 4$-matrix in (\ref{eq:PPprime}).
 For a generic $x$ in $\PP^3$, the ideal \eqref{eq:system} has exactly $\alpha$ complex zeros in $\PP^5$.
  To compute the focal locus, we treat
  the coordinates of $x$ as parameters, and 
  we add to \eqref{eq:system} the $5 \times 5$ minors of
  the Jacobian of (\ref{eq:system}) with respect to the Pl\"ucker coordinates.
  This gives us an ideal in $\CC[p_{01},\ldots,p_{23},x_0,\ldots,x_3]$.
By saturating and eliminating $p_{01},\ldots,p_{23}$, we obtain the 
ideal in $\CC[x_0,x_1,x_2,x_3]$ that defines the focal locus $\mathcal{F}(C)$ in $\PP^3$.

\section{The Concurrent Lines Variety}
\label{sec3}

We next investigate the conditions for multiple lines to be
all concurrent in a single point. This will be applied in Section \ref{sec5} to systems of
geometric cameras.  The {\em concurrent lines variety} $V_n$ consists
of ordered $n$-tuples of lines in $\PP^3$ that meet in a point $x$. 
The lines containing a fixed $x$ form a linear
space of constant dimension~$2$ (the $\alpha$-plane for
$x$). From this one infers that $V_n$ is irreducible of dimension
$2n+3$, provided $n \geq 2$. Let $I_n$ denote the prime ideal of $V_n$
in the polynomial ring of $6n$ Pl\"ucker coordinates. We regard
$V_n=V(I_n)$ as a subvariety in the product of projective spaces
$(\PP^5)^n$.  Hence its ideal $I_n$ is $\ZZ^n$-graded.

\smallskip

The following result fully characterizes the prime ideal $I_n$ of the concurrent lines variety.

\begin{theorem} \label{thm:conclines}
Let $P_1,P_2,\ldots,P_n$ be skew-symmetric $4 {\times} 4$-matrices
of unknowns that represent lines in $\PP^3$,
and let $P_1^*, P_2^*, \ldots, P_n^*$ be the dual matrices.
The ideal $I_n$ is minimally generated
by the $\binom{n+1}{2}$ quadrics ${\rm trace}(P_i P_j^*)$ and the 
 $10 \binom{n}{3}$ cubics  obtained as $3 \times 3$-minors~of
  $\bigl(P_1 u , P_2 u, \ldots, P_n u \bigr)$
  where $u $ runs over $\{e_1,e_2,e_3,e_4,
  e_1{+}e_2, e_1{+}e_3, \ldots,e_3 {+} e_4 \}$.
For the reverse lexicographic order,  the reduced Gr\"obner basis  of $I_n$
consists of   $\binom{n+1}{2}$ quadrics,
$12 \binom{n}{3}$ cubics and $ 4 \binom{n+1}{4}$ quartics.
Their leading terms are squarefree, so the initial ideal is~radical.
\end{theorem}

Our proof rests on computations with the computer algebra system
{\tt Macaulay2}~\cite{M2}.

\begin{proof}
The case $n=2$ is easy. We begin
with $n=3$. Let $P, Q, R$ be skew-symmetric $4 \times 4$-matrices representing three lines.
These matrices have rank $2$.
 The Pl\"ucker quadrics~are
\begin{equation}
\label{eq:Q1}
{\rm trace}(PP^*) \,\,\, = \,\,\,
{\rm trace}(QQ^*) \,\,\, = \,\,\,
 {\rm trace}(RR^*) \,\,\,  =\,\,\, 0.
\end{equation}
Furthermore,  the three lines are pairwise concurrent if and only if
\begin{equation}
\label{eq:Q2}
{\rm trace}(PQ^*) \,\,\, = \,\,\,
{\rm trace}(PR^*) \,\,\, = \,\,\,
 {\rm trace}(QR^*) \,\,\,  =\,\,\, 0.
\end{equation}

Using a computation with {\tt Macaulay2}, we find that
the ideal generated by the six quadrics in (\ref{eq:Q1}) and (\ref{eq:Q2})
is radical. It is the intersection of two prime ideals, each minimally generated by
ten cubics in addition to (\ref{eq:Q1}) and (\ref{eq:Q2}).
The first prime represents triples of lines that are coplanar.
This is an extraneous component for us. The second prime is
the concurrent lines variety. The cubic generators of that second prime ideal are the $3 \times 3$-minors
of the $4 \times 3$-matrix $ (P u, Q u , R u)$,
where $u$ is a column vector in $\RR^4$. These span a ten-dimensional
space of cubics. A basis for that space is obtained by selecting the vector $u$ from the set
\begin{equation}
\label{eq:ourten} \bigl\{\,e_1\,,\,\,e_2\,,\,\,e_3\,,\,\,e_4\,,\,\,
  e_1 + e_2\,,\, \,e_1 + e_3 \,,\, \, e_1+ e_4 \,,\, \,
  e_2 + e_3 \,,\, \, e_2 + e_4 \,,\, \, e_3 + e_4\, \bigr\}.
  \end{equation}
We note that the cubics for coplanar triples of lines are the $3 \times 3$-minors
of the $4 \times 3$-matrix $ (P^* \cdot u, Q^* \cdot u , R^* \cdot u)$,
where $u \in \RR^4$. A basis of $10$ cubics is obtained from the same set (\ref{eq:ourten}).

Using {\tt Macaulay2}, we now
compute the reduced Gr\"obner basis of our prime ideal from the $6+10 = 16$ generators
with respect to the reverse lexicographic order determined by
$$ p_{01} {>} p_{02} {>} p_{03} {>} p_{12} {>} p_{13} {>} p_{23} > q_{01} {>} q_{02} {>}
q_{03}  {>} q_{12} {>} q_{13} {>} q_{23} > r_{01} {>} r_{02} {>} r_{03} {>} r_{12} {>} r_{13} {>} r_{23}. $$
The initial monomial ideal is generated by the leading terms in the reduced Gr\"obner basis:
\begin{equation}
\label{eq:M3}
\begin{matrix}
M_{3} & = \!\!\!\! & \!\!
\bigl\langle\,\,
p_{03} p_{12},\, q_{03} q_{12},\, r_{03} r_{12}\,,\,\,\,
p_{23} q_{01},\, p_{23} r_{01},\, q_{23} r_{01},\,
\\
&& 
\,      p_{12} q_{02} r_{01}\,,\,\,
      p_{12} q_{03} r_{01}\,,\,\,
      p_{12} q_{03} r_{12}\,, \,\,
      p_{12} q_{03} r_{02},\, \\ &&
 \,     p_{13} q_{02} r_{01}\,,\,\,
      p_{13} q_{03} r_{01}\,, \, \,
      p_{13} q_{03} r_{12}\,,\,\,
      p_{13} q_{03} r_{02}\,,\,\,  \\ &&
\,      p_{23} q_{03} r_{02}\,,\,\,
      p_{23} q_{03} r_{12}\,,\, \,
      p_{23} q_{13} r_{02}\,,\,\,
      p_{23} q_{13} r_{12}\,,\,\,
\\      
&& \qquad \qquad\,\,
 p_{12} q_{02} q_{13} r_{02}\,,\,\,
 p_{12} q_{02} q_{13} r_{12}\,,\,\,
 p_{13} q_{02} q_{13} r_{02}\,,\,\,
 p_{13} q_{02} q_{13} r_{12}\,\bigr\rangle.
\end{matrix}
\end{equation}
This shows that the reduced Gr\"obner basis consists of
 $6$ quadrics, $12$ cubics and $4$ quartics. All $22$ leading terms are squarefree.
      This completes the proof of Theorem~\ref{thm:conclines} for $n  = 3$. 
      
We next consider the case $n=4$. A {\tt Macaulay2} computation
verifies that Theorem~\ref{thm:conclines} is true here.
The ideal $I_4$ is minimally generated by
the $10$ quadrics $\,{\rm trace}(P_iP_j^*)$ together with 
 $40 = 10 \binom{4}{3}$ cubics, namely the
$10$ cubics from $I_3$ for any three of the four lines.
The initial ideal $M_4 = {\rm in}(I_4)$ is minimally 
generated by $10 $ quadratic monomials,
$48 = 12 \binom{4}{3}$ cubic monomials,
and $20 = 4 \binom{5}{4}$ quartic monomials.
The quadrics and cubics come from $M_3$ for
any three of the four lines. Among the quartics are
the $16 = 4 \binom{4}{3}$ quartics from $M_3$ for any three of the four lines.
However, the reduced Gr\"obner basis of $I_4$ now also contains four quadrilinear forms.
These contribute four new generators of the monomial ideal~$M_4$:
\begin{equation}
\label{eq:newquartics}
p_{12} q_{02} r_{13} s_{02}\,,\,\,
p_{12} q_{02} r_{13} s_{12}\,,\,\,
p_{13} q_{02} r_{13} s_{12}\,,\,\, 
p_{13} q_{02} r_{13} s_{02}.
\end{equation}

We next assume $n \geq 5$. We write $\mathcal{G}_n$ for the union of the
various reduced Gr\"obner bases, obtained from $I_4$ for any four of the $n$ lines.
The set $\mathcal{G}_n$ has $\binom{n+1}{2}$ quadrics
${\rm trace}(P_i P_j^*)$, and it has  $12 \binom{n}{3}$ cubics, 
namely those having the $12$ leading terms in (\ref{eq:M3}), for any three lines.
Finally, there are  $4 \binom{n+1}{4} =  4 \binom{n}{3} + 4 \binom{n}{4}$
quartics in $\mathcal{G}_n$. Their leading monomials are the quartics in (\ref{eq:M3}), for any three lines,
and the quartics in (\ref{eq:newquartics}), for any four of the $n$ lines.

We claim that $\mathcal{G}_n$ is the reduced Gr\"obner basis for the ideal
$\langle \mathcal{G}_n \rangle$ it generates. This can be verified computationally
with {\tt Macaulay2} for $n \leq 7$. For $n \geq 8$, we argue as follows.
Consider any two polynomials in $\mathcal{G}_n$. We must show that their
S-polynomial reduces to zero upon division with respect to $\mathcal{G}_n$.
If their leading monomials are relatively prime then this is automatic, by
Buchberger's First Criterion. Otherwise, the leading monomials have a 
Pl\"ucker variable in common. This means that at most seven of the $n$ lines
are involved in the two polynomials. But then their S-polynomial reduces to zero
because the Gr\"obner basis property is already known for $n \leq 7$. A similar argument shows
that no trailing term in $\mathcal{G}_n$ is a multiple of an leading term. Hence
$\mathcal{G}_n$ is the reduced Gr\"obner basis for its ideal.

The minimal generators of the ideal $\langle \mathcal{G}_n \rangle$ are obtained from the
minimal generators of $I_4$, for any four of the $n$ lines. Hence $\langle \mathcal{G}_n \rangle$
is generated by the  $\binom{n+1}{2}$ quadrics and the
$10 \binom{n}{3}$ cubics that are listed in the statement of Theorem~\ref{thm:conclines}.
Its leading terms are square-free.

We must prove that the ideal $\langle \mathcal{G}_n \rangle$ equals
the ideal $I_n$ we are interested in. By construction, all generators of $\mathcal{G}_n$
vanish on the concurrent lines variety $V_n = V(I_n)$. Therefore,
\begin{equation}
\label{eq:inclusion}
 \langle \mathcal{G}_n \rangle \,\,\subseteq \,\, I_n . 
\end{equation}
Moreover, the initial ideal of $\langle \mathcal{G}_n \rangle$ is radical,
and hence $\langle \mathcal{G}_n \rangle$ is a radical ideal. 
To complete the proof, all we now need is that the set
$\mathcal{G}_n$ cuts out the variety $V_n$
set-theoretically.  This is equivalent to the statement that $n \geq 4$ distinct 
lines in $\PP^3$ are concurrent if and only if any three of the $n$ lines are concurrent.
This is indeed the case. 
\end{proof}

\begin{remark} \rm Suppose all $P_i$ satisfy the Pl\"ucker constraint ${\rm trace}(P_i P_i^*)=0$.
The  four $3 \times 3$-minors~of $\bigl(P_i u , P_j u, P_k u \bigr)$ are scalar multiples of a single trilinear polynomial $T_u$ that expresses the condition for the planes $u \vee p_i, u \vee p_j$ and $ u \vee p_k$ to be linearly dependent, i.e., for $p_i, p_j, p_k$ to admit a {\em transversal line} passing through $u$. In fact, three lines are concurrent if and only if they are pairwise coplanar and they admit a transversal not contained in the planes
defined by any two of them \cite{TPH}. From this we deduce that $V_n$ is cut out set-theoretically by $\binom{n+1}{2}$ bilinear quadrics ${\rm trace}(P_i P_j^*)$ and the 
 $4 \binom{n}{3}$ trilinear cubics $T_u$ where $u$ runs over only $\{e_1,e_2,e_3,e_4\}$. This is confirmed by computation with {\tt Macaulay2}.
\end{remark}

The concurrent lines variety $V_n$  has
codimension $3n-3$ in $(\mathbb{P}^5)^n$. 
Its class $[V_n]$ in the cohomology ring of $(\PP^5)^n$
is a homogeneous polynomial of degree $3n-3$ in $n$ unknowns
$t_1,t_2,\ldots,t_n$, where $t_i$ represents the hyperplane
class in the $i$-th factor $\mathbb{P}^5_i$.
In the language of commutative algebra, $[V_n]$ is 
known as the {\em multidegree} of $V_n$. We
refer to \cite[Section 8.5]{MS} for an introduction to multidegrees.
We also note that there is a built-in command {\tt multidegree} in
{\tt Macaulay2}  for computing $[V_n]$ from  the ideal $I_n$.
Using this command, we found experimentally that
the multidegree of the concurrent lines variety  is the polynomial
\begin{equation}
\label{eq:multidegree}
 [V_n] \,\,\,\, = \,\,\,\,
(t_1 t_2 t_3 \cdots t_n)^3 \cdot \bigl(\, 4 \sum_{(i,j)} t_i^{-2} t_j^{-1} \,\, + \,\,
8 \sum_{\{i,j,k\}} t_i^{-1} t_j^{-1} t_k^{-1}\, \bigr). \end{equation}
The first sum  is over ordered pairs $(i,j)$ with $i \not= j$. The second sum is over unordered triples $\{i,j,k\}$.
The sum of the coefficients of $[V_n]$ equals $\,8\binom{n+1}{3}$.
The variety of $M_n$ decomposes into components
  $(\PP^2)^{n-2} {\times} \PP^4_i {\times} \PP^3_j$  and
    $(\PP^2)^{n-3} {\times} \PP^3_i {\times} \PP^3_j {\times} \PP^3_k$.
    These are  recorded by~$[V_n]$.

After completion of this article,  Laura Escobar and Allen Knutson \cite{EK} found
a proof for the formula (\ref{eq:multidegree}).
Their derivation in \cite{EK} rests on methods from representation theory.

\begin{example} \label{ex:compo4} \rm 
Let $n=4$. The multidegree for four concurrent lines equals
\begin{small}
$$ \begin{matrix} [V_4] & = &
4 t_1^3 t_2^3 t_3^2 t_4^1 + 
4 t_1^3 t_2^3 t_3^1 t_4^2  + 
4 t_1^3 t_2^2 t_3^3 t_4^1 + 
4 t_1^3 t_2^1 t_3^3 t_4^2 + 
4 t_1^3 t_2^2 t_3^1 t_4^3  + 
4 t_1^3 t_2^1 t_3^2 t_4^3  +  
4 t_1^2 t_2^3 t_3^3 t_4^1  + 
4 t_1^1 t_2^3 t_3^3 t_4^2  + \\ & & 
4 t_1^2 t_2^3 t_3^1 t_4^3  + 
4 t_1^1 t_2^3 t_3^2 t_4^3  +
4 t_1^2 t_2^1 t_3^3 t_4^3  + 
4 t_1^1 t_2^2 t_3^3 t_4^3 +  \
8 t_1^3 t_2^2 t_3^2 t_4^2 + 
8 t_1^2 t_2^3 t_3^2 t_4^2 + 
8 t_1^2 t_2^2 t_3^3 t_4^2 + 
8 t_1^2 t_2^2 t_3^2 t_4^3 .
\end{matrix}
$$
\end{small}
The first term in the multidegree represents the following four minimal primes of $M_4$:
$$ V(p_{12}, p_{13}, p_{23}, q_{12}, q_{13}, q_{23} ,\rho, r_{23}, \sigma) \,
\simeq \, \PP^2 \times \PP^2  \times \PP^3 \times \PP^4, 
\qquad (\rho,\sigma) \in \{r_{03}, r_{12}\} \times \{s_{03} , s_{12}\}.
$$
The last term in the multidegree represents the following eight minimal primes of $M_4$:
$$ V(p_{23},\pi, q_{03} ,\phi, r_{01}, \rho, s_{01}, s_{02}, s_{12})\simeq
\PP^3 \times \PP^3 \times \PP^3 \times \PP^2, 
\,(\pi,\phi,\rho) \in \{p_{03} {,} p_{12} \} \times   \{q_{02} {,} q_{13}\} \times
\{r_{03}{,}  r_{12}\}  .$$
All other irreducible components are similar. Each of the $80 = 8 \binom{5}{3}$ components 
is a product of projective spaces, defined by the vanishing of nine Pl\"ucker coordinates in
 $(\PP^5)^4$. $\diamondsuit$
\end{example}                                 

\section{Rational Cameras}
\label{sec4}

Let $C$ be a congruence of bidegree $(1,\beta)$. The {\em rational camera} defined by $C$ is the map
\begin{equation}\label{eq:cong_proj} \PP^3 \dashrightarrow C \subset {\rm Gr}(1,3)
\end{equation}
that associates a generic point $x$ in $\PP^3$ with the unique line in $C$ that passes through $x$. 
This map is defined everywhere except at the focal locus. We already noted that the focal locus of $C$
equals the fundamental locus, and its dimension is either zero or one.
We write $C(x)$ for the  image of $x$ under the map (\ref{eq:cong_proj}).
The point in $\PP^5$ that represents the line $C(x)$  in the Pl\"ucker embedding of ${\rm Gr}(1,3)$ 
is the intersection of  $C$ with the $\alpha$-plane associated with~$x$.

In this section we discuss the classification of order one congruences $C$,
and we derive some explicit formulas for the  rational maps $x \mapsto C(x)$.
We begin with the two easiest cases, pinhole and two-slit cameras, where the congurences are obtained by intersecting the Pl\"ucker quadric with linear spaces. We then move on to study rational cameras in full generality.

\subsection{Pinhole and Two-Slit Cameras}

If $C$ is a $(1,0)$-congruence, then $C$ is an $\alpha$-plane for some point $c$ in $\PP^3$, and \eqref{eq:cong_proj} represents a traditional pinhole camera. The 
image of a point $x$ is the line with Pl\"ucker coordinates 
\begin{equation}\label{eq:pinhole}
C(x) \,\, = \,\, x \vee c \,= \,\,\begin{bmatrix}
c_0 x_1 - c_1 x_0\\
c_0 x_2 - c_2 x_0\\
c_0 x_3 - c_3 x_0\\
c_1 x_2 - c_2 x_1\\
c_1 x_3 - c_3 x_1\\
c_2 x_3 - c_3 x_2\\
\end{bmatrix} \,\in \,{\rm Gr}(1,\PP^3). \end{equation}
There is a complete symmetry between the center $c$ and the projected point $x$, and if we write $C_c$ and $C_x$ for the $\alpha$-planes of lines through $c$ and $x$ respectively, then $\{x \vee c\} = C_c \cap C_x$.

Next, we consider a congruence $C$  that is defined
by two general linear forms in the six Pl\"ucker coordinates on
${\rm Gr}(1,\PP^3)$. Then $C$ is a  $(1,1)$-congruence. 
The line of all linear forms that vanish on $C$ intersects
the dual Grassmannian ${\rm Gr}(1,(\PP^3)^*) \subset (\PP^5)^*$ in two points
$p^*$ and $q^*$. The congruence is hence defined
by the corresponding lines, i.e.~we have
$\,C = \{ r \in {\rm Gr}(1,\PP^3)\,: \,p \vee r = q \vee r = 0 \}$.
We denote the primal Pl\"ucker coordinates of the two lines by
$$ p = ( p_{01} : p_{02} : p_{03} : p_{12} : p_{13} : p_{23} ) \quad \hbox{and} \quad
q = ( q_{01} : q_{02} : q_{03} : q_{12} : q_{13} : q_{23} ) . $$
Geometrically, the congruence $C$ is the family of 
 common transversals to $p$ and $q$.
 Each point of $\PP^3$ outside these two lines lies on a unique such transversal.
 Hence
the focal locus  $\mathcal{F}(C)$ is the union of the two lines $p$ and $q$.
 The associated rational camera is a {\em two-slit camera} \cite{BGP}. 
Note that any two skew lines $p$ and $q$ in $\PP^3$ define such a congruence of bidegree $(1,1)$.

Given a general point $x$ in $\PP^3$, the plane containing
$x$ and the line $p$  is the point in $(\PP^3)^*$ with coordinates
$x \vee p = Px$. Likewise, $x \vee q = Qx$ is the plane spanned by
the point $x$ and the line $q$. Here $P$ and $Q$ are
the skew-symmetric $4 \times 4$-matrices that represent $p$ and $q$.
Intersecting these two planes gives the line in the congruence that contains $x$.
In symbols,
\begin{equation}\label{eq:two_slit}
C(x) \,\,\,= \,\,\, (x \vee p) \wedge (x \vee q)\,\,=\,\, Pxx^TQ-Qxx^TP.
\end{equation}
The coordinates of the Pl\"ucker vector $C(x)$ are quadratic in the coordinates of $x$,
and they are bilinear in $(p,q)$. For instance, the first coordinate of $C(x)$, indexed by $01$, is equal to
$$ \begin{matrix} (q_{13} p_{23}-q_{23} p_{13}) x_0 x_2
+(q_{23} p_{12}-q_{12} p_{23}) x_0 x_3
+(q_{23} p_{03}-q_{03} p_{23}) x_1 x_2
+(q_{02} p_{23}-q_{23} p_{02}) x_1 x_3 \\
+(q_{03} p_{13}-q_{13} p_{03}) x_2^2
+(q_{12} p_{03}-q_{03} p_{12}-q_{02} p_{13}+q_{13} p_{02}) x_2 x_3
+(q_{02} p_{12}-q_{12} p_{02}) x_3^2.
\end{matrix}
$$
In summary, the picture of $x$ taken with the two-slit camera $C$ is the line given by (\ref{eq:two_slit}).

\begin{example}[Pushbroom cameras] \rm A {\em pushbroom camera}
  \cite{GH} is a device consisting of a linear array of sensors
  mounted on a platform that can move along a line perpendicular to
  the sensors. As the platform moves, the camera scans a family of
  viewing planes. This type of optical system is commonly used in
  aerial and satellite cameras as well as CT systems.

 It was observed in \cite{Pon} that pushbroom cameras are two-slit cameras where one of the two
 slits lies on the plane at infinity.  If we identify Euclidean 3-space with the affine chart $U_0  = \{x_0 \ne 0\}$ 
then  $q$ can be any line of the form
 $q  = (0:0:0:q_{12}:q_{13}:q_{23})$. A standard choice is the line at infinity
 that is orthogonal to $p$, with respect to the
   usual scalar product on $U_0 \simeq \RR^3$.
   That line has the Pl\"ucker coordinates $q= ( 0 : 0 : 0 : p_{03} : -p_{02} : p_{01} )$.
For this choice of $q$, the polynomial formula \eqref{eq:two_slit} for
the image line $C(x)$ specializes to
$$
{\scriptsize 
\begin{bmatrix}
 -p_{02} p_{12} x_0^2-p_{03} p_{13} x_0^2+p_{02}^2 x_0 x_1+p_{03}^2 x_0 x_1-p_{01} p_{02} x_0 x_2-p_{01} p_{03} x_0 x_3 \\
 p_{01} p_{12} x_0^2-p_{03} p_{23} x_0^2-p_{01} p_{02} x_0 x_1+p_{01}^2 x_0 x_2+p_{03}^2 x_0 x_2-p_{02} p_{03} x_0 x_3\\
p_{01} p_{13} x_0^2+p_{02} p_{23} x_0^2-p_{01} p_{03} x_0 x_1-p_{02}p_{03} x_0 x_2+p_{01}^2 x_0 x_3+p_{02}^2 x_0 x_3\\
p_{01} p_{12} x_0 x_1-p_{03} p_{23} x_0 x_1-p_{01}p_{02}x_1^2+p_{02}p_{12} x_0 x_2+p_{03}p_{13}x_0 x_2{+}p_{01}^2x_1x_2{-}p_{02}^2 x_1 x_2{+}p_{01} p_{02} x_2^2{-}p_{02} p_{03}x_1 x_3{+}p_{01}p_{03}x_2 x_3\\
p_{01} p_{13} x_0 x_1+p_{02} p_{23} x_0 x_1-p_{01} p_{03} x_1^2-p_{02} p_{03} x_1 x_2+p_{02} p_{12} x_0 x_3{+}p_{03} p_{13} x_0 x_3{+}p_{01}^2 x_1 x_3{-}p_{03}^2 x_1 x_3{+}p_{01}p_{02}x_2 x_3{+}p_{01}p_{03}x_3^2\\
p_{01} p_{03} x_1 x_2-p_{01} p_{13} x_0 x_2-p_{02} p_{23} x_0 x_2+p_{02
}p_{03} x_2^2+p_{01} p_{12} x_0 x_3{-}p_{03} p_{23} x_0 x_3{-}p_{01} p_{02} x_1 x_3 {-}p_{02}^2 x_2 x_3{+}p_{03}^2 x_2 x_3{-}p_{02} p_{03} x_3^2
\end{bmatrix}
}.
$$
This Pl\"ucker vector represents the picture of the point $x$ taken by the pushbroom camera.~$\diamondsuit$
\end{example}

The $(1,1)$-congruences $C$ we consider are defined over the
real numbers.  From the perspective of real algebraic
geometry, one distinguishes the following three possibilities for
the focal locus.  The two lines $p$ and $q$ in  $\mathcal{F}(C)$
may be real and distinct, real and coincide (when the line
of linear forms defining $C$ intersects ${\rm Gr}(1,\PP^3)$ in a double
point), or they may form a complex conjugate pair of lines.
  In the first case, the $(1,1)$-congruence $C$ is {\em
  hyperbolic}.  This includes the pushbroom cameras. In the second
case, $C$ is said to be {\em parabolic}, and consists of a
one-parameter family of flat pencils of lines centered on the line
$p=q$. In the last case, the focal locus $\mathcal{F}_C$ has no real
points, and the $(1,1)$-congruence $C$ is said to be {\em elliptic}. We refer to~\cite{BGP}
for a more detailed presentation of the real geometry of linear cameras.

\subsection{Congruences of Order One and Higher Class}

We now consider $(1,\beta)$-congruences for any $\beta$. These 
were classified in 1866 by Kummer~\cite{Kum}. His result
was then refined and extended by various authors in the 20th century.
The following version was derived by De Poi in \cite{DeP}.
We refer to his article for more information.

\begin{theorem}\label{th:classify_cong}
Let  $\,C$ be a $(1,\beta)$-congruence with focal locus $\mathcal F(C)$.
Then one of the following four situations is the case:
\begin{enumerate}
\item $\mathcal F(C)$ is a point $c$, and $C$ is the $\alpha$-plane of lines through $c$. Here $\beta = 0$.

\item $\mathcal F(C)$ is a twisted cubic in $\PP^3$, and $C$ consists of its secant lines.
Here $\beta = 3$.
\item $\mathcal F(C)$ is the union of a rational curve $X$ of degree $\beta$ and a line $L$ that intersects $X$ in $\beta-1$ points. The congruence $C$ is the family of lines that intersects both $L$ and $X$.
Here we allow for degenerate cases:  the points in $X \cap L$ are counted with multiplicity.
\item $\mathcal F(C)$ is (a non-reduced) line $L$. The congruence $C$ is described by a 
morphism $\phi$ of degree $\beta > 0$ from $L^*$ to $L$, where $L^*$ denotes the planes containing $L$: a line is in $C$ if it belongs to a pencil of lines lying in a plane $\Pi$ in $L^*$ and passing through $\phi(\Pi)$.
\end{enumerate}
\end{theorem}

\noindent We next describe the rational cameras (\ref{eq:cong_proj})
for each of these families of congruences.

\medskip

\noindent {\bf Type 1: $\mathcal F(C)$ is a point.} This is the pinhole camera described in Section 4.1.

\medskip

\noindent {\bf Type 2: $\mathcal F(C)$ is a twisted cubic.} 
After a change of coordinates,  the twisted cubic in $\PP^3$ is the image of the map
 $\,( s : t ) \mapsto ( s^3 : s^2t : st^2 : t^3)$.
 The corresponding rational camera~is
 \begin{small}
 \begin{equation}\label{eq:proj_cubic}
C(x) \,\,\, = \,\,\,
\begin{bmatrix}
(x_0 x_2-x_1^2)^2\\
(x_0x_2-x_1^2) (x_0 x_3-x_1x_2)\\
x_0 x_2^3+x_1^3 x_3 -3 x_0 x_2 x_1x_3+x_0^2x_3^2\\
(x_1 x_3-x_2^2)(x_0x_2-x_1^2)\\
(x_1 x_3-x_2^2)(x_0x_3-x_1x_2)\\
(x_1 x_3-x_2^2)^2
\end{bmatrix}.
\end{equation}
\end{small}
The ideal of the congruence $C$ is generated by the Pl\"ucker relation together with five quadrics
\begin{equation}\label{eq:secant_cubic}
p_{13}^2-p_{03} p_{23}-p_{12} p_{23}, \,\,\, p_{12} p_{13}-p_{02} p_{23}, \,\,\, p_{12}^2-p_{01} p_{23},
  \,\,\,   p_{02} p_{12}-p_{01} p_{13}, \,\,\, p_{02}^2-p_{01} p_{03}-p_{01} p_{12}.
\end{equation}
If we augment this ideal
by the four entries of  $P x$, where $x= (x_0 : x_1 : x_2 : x_3)$ 
is an unknown world point in $\PP^3$, then the radical of the resulting ideal is generated by  the quadrics in \eqref{eq:secant_cubic} together with six bilinear equations that can be written in matrix-vector form as follows:
\begin{equation}
\label{eq:66linsys}
{\small
\begin{bmatrix}
0 & 0 & 0 & x_3 & -x_2 & x_1\\
0 & 0 & 0 & x_2 &-x_1 & x_0\\
0 & x_3 & -x_2 & 0 & 0 & x_0\\
0 & x_2 & -x_1 &-x_1 & x_0 &0\\
x_3 & 0 & -x_1 & 0 & x_0 & 0\\
x_2 & -x_1 & 0 & x_0 & 0 & 0
\end{bmatrix}
\begin{bmatrix}
p_{01}\\
p_{02}\\
p_{03}\\
p_{12}\\
p_{13}\\
p_{23}
\end{bmatrix}=
\begin{bmatrix}
0\\ 0 \\ 0 \\ 0 \\ 0 \\ 0
\end{bmatrix}.
}
\end{equation}
This $6\times 6$ matrix has rank $5$. The solution space of (\ref{eq:66linsys})
is spanned by the vector  in (\ref{eq:proj_cubic}). Inside $\PP^5$, the secant congruence of the twisted cubic is a Veronese surface \cite[Proposition~1]{DeP}. 

The {\em twisted cubic camera} (\ref{eq:proj_cubic}) has a nice interpretation
in terms of tensor decompositions. For this, we identify $\PP^3$ with the
space  of symmetric $2 \times 2 \times 2$-tensors. We 
seek to decompose an arbitrary tensor as 
the sum of two rank $1$ tensors. Equivalently, we seek to write
a binary cubic $x_0 u^3 + 3 x_1 u^2 v + 3 x_2 u v^2 + x_3 v^3$
as the sum of two cubes of linear forms in $u$ and $v$.
Rank~$1$ tensors are points on the twisted cubic curve.
The desired representation is unique.
It is given by the
intersection points of the twisted cubic with the secant line $C(x)$.


\medskip

\noindent {\bf Type 3: $\mathcal F(C)$ is a rational curve $X$ and a line $L$.} 
After a change of coordinates we may assume that the line is
$L = \{ (0:0:x_2:x_3) \in \PP^3 \,:\, (x_2:x_3) \in \PP^1 \} $.
The dual line $L^*$ parametrizes all planes in $\PP^3$ that contain $L$.
A natural parametrization  $\PP^1 \rightarrow L^*$ is given by
identifying $(x_0:x_1)$ with the plane in $\PP^3$ with dual coordinates $(x_1,-x_0,0,0)$.

To build our rational camera, we take an arbitrary
 rational curve $X$ of degree $\beta$ that intersects $L$
in $\beta-1$ points. Each such curve $X$ is given by a parametric representation 
\begin{equation}
\label{eq:Xpara}
 \PP^1 \rightarrow X , \,(s:t) \mapsto \bigl( s f(s,t): t f(s,t): g(s,t): h(s,t) \bigr ), 
 \end{equation}
where $f,g$ and $h$ are arbitrary binary forms of degree $\beta-1$, $\beta$ and $\beta$
respectively.

\begin{proposition} For the rational camera of Type 3, the map (\ref{eq:cong_proj}) is
given by
\begin{equation}
\label{eq:type3C}
 C(x) \,\,= \,\, \begin{bmatrix} x_0 \\ x_1 \\ x_2 \\ x_3 \end{bmatrix} 
\, \vee \,\begin{bmatrix} x_0 f(x_0,x_1) \\ x_1 f(x_0,x_1) \\ g(x_0,x_1) \\ h(x_0,x_1) \end{bmatrix} .
\end{equation}
\end{proposition}

\begin{proof}
The two column vectors in (\ref{eq:type3C}) represent two points in $\PP^3$
that lie on the plane in $L^*$ with coordinates $(x_0 : x_1)$, according to the
parametrization above.
We see from (\ref{eq:Xpara}) that
the second point lies on the curve $X$.
Hence  (\ref{eq:type3C}) is a line
the intersects both $L$ and $X$.
\end{proof}

\begin{remark} \label{rem:Chow} \rm
We now describe the ideal of the congruence $C$ 
in the coordinate ring of the Grassmannian ${\rm Gr}(1,\PP^3)$.
To do this, we use the concept of Chow forms, as described in \cite{DS}.
Recall that {\em Chow form} ${\rm Ch}_Z$ of an irreducible curve $Z$ of degree $\gamma$ in $\PP^3$ is a
 hypersurface of degree $\gamma$ in the Grassmannian ${\rm Gr}(1,\PP^3)$. Its points are 
all the lines in $\PP^3$ that intersect $Z$.

With this notation, the ideal of $C$ is the saturation of
$ \langle {\rm Ch}_{L} , {\rm Ch}_{X} \rangle  $
with respect to
$ \cap_{i=1}^{\beta-1} \langle P u_{i} \rangle $, where  $u_{i}$ are the intersections between $L$ and $X$, and $Ch_L$ and $Ch_X$ are the  Chow forms of $L$ and $X$ respectively. Hence the ideal $\langle {\rm Ch}_{L} , {\rm Ch}_{X} \rangle$
represents all lines that intersect both $L$ and $X$. The saturation removes $\beta-1$ extraneous components, namely the $(1,0)$-congruences of lines passing through the points $u_{i}$.
 We conjecture that the resulting ideal is generated by the Pl\"ucker quadric,
the linear Chow form ${\rm Ch}_L$, and $\beta$ linearly independent forms of degree~$\beta$ (including ${\rm Ch}_X$). This description was observed experimentally.
\end{remark}

\begin{example} \label{ex:type3} \rm
Fix $\beta = 3$ and let $X$ be the twisted cubic curve
given as in (\ref{eq:Xpara}) with
$\,f=(s-t)(s+t), \,g = s^3 $ and $ h=t^3$. The ideal of $X$ is generated by the
$2 \times 2$-minors of
\begin{equation}
\label{eq:twobythree} \begin{pmatrix}
     x_1+x_3 & x_2-x_0 & x_3    \\
       x_2  & x_1+x_3 & x_2-x_0  \\
\end{pmatrix} . 
\end{equation}
The line $L = V(x_0,x_1)$ meets the curve $X$ in the two points
$(0:0:1:1)$ and $(0:0:1:-1)$.
 The corresponding $(1,3)$-congruence $C$ is
 parametrized by (\ref{eq:type3C}).
  The ideal of $C$ equals
\begin{equation}
\label{eq:type3ideal} \begin{matrix}
\bigl\langle \,p_{01}\,, \, p_{03} p_{12} -p_{02} p_{13}\,, \,
   p_{02} p_{03}^2-p_{12}^2 p_{13} - p_{02} p_{03} p_{23}+p_{12} p_{13} p_{23}, \,\,\,\, \\ \quad
 p_{03}^3-p_{12} p_{13}^2-p_{03}^2 p_{23}+p_{13}^2 p_{23}\,, \,   
      p_{02}^2 p_{03} - p_{12}^3 - p_{02}^2 p_{23} + p_{12}^2 p_{23}    \bigr\rangle.
      \end{matrix} 
      \end{equation}
\end{example}

\smallskip

\noindent {\bf Type 4: $\mathcal F(C)$ is a non-reduced line $L$.} 
This is the degenerate case of Type 3 congruences when the binary form $f $ 
is identically zero.
The degree $\beta$ morphism $\phi : L^* \rightarrow L$
promised in Theorem~\ref{th:classify_cong}
sends  $(x_0:x_1)$ to the point $(0:0:x_2:x_3) = \bigl(0:0:g(x_0,x_1) : h(x_0,x_1 )\bigr)$
on the line $L \subset \PP^3$. The corresponding rational camera is given by
the formula (\ref{eq:type3C}) with $f=0$.

\begin{example} \label{ex:type4} \rm
Let $\beta = 3$ as in Example \ref{ex:type3} but now
with $f = 0, \,g =s^3$ and $h = t^3$.
The non-reduced structure of $L$ is the ideal
$\langle x_0^2, x_0 x_1, x_1^2  \rangle $, obtained from (\ref{eq:twobythree}) by setting $x_3=0$.
The resulting $(1,3)$-congruence $C$ is a toric surface
in ${\rm Gr}(1,\PP^3) \subset \PP^5$. Its prime ideal equals
 $$ 
\bigl\langle \,p_{01}\,, \, p_{03} p_{12} -p_{02} p_{13}\,, \,\,
   p_{02} p_{03}^2-p_{12}^2 p_{13} \,,\,\,
   p_{03}^3-p_{12} p_{13}^2 \,,\,\,
      p_{02}^2 p_{03} - p_{12}^3 \,
         \bigr\rangle.
         $$
Note that the three binomial cubics are the initial forms of the three cubics in           (\ref{eq:type3ideal}).
\hfill $\diamondsuit$
\end{example}

\section{Multi-Image Varieties}
\label{sec5}

In this section, we use the concurrent lines variety $V_n$ from Section \ref{sec3} to 
characterize multi-view correspondences for $n$ rational cameras. 
We fix congruences $C_1,\ldots, C_n \subset {\rm Gr}(1,\PP^3)$,
where $C_i$ has bidegree $(1,\beta_i)$
for some $\beta_i \in \mathbb{N}$.
Combining  their maps as in  (\ref{eq:cong_proj}) gives
\begin{equation}
\label{eq:mapCCC}
\PP^3 \dashrightarrow C_1 \times \cdots \times C_n \,\, ,\,\,\, x \mapsto \bigl(C_1(x),\ldots,C_n(x) \bigr).
\end{equation}
The base locus of this rational map is the product of the focal loci,
$\mathcal{F}(C_1) \times \cdots \times \mathcal{F}(C_n)$.
We define the {\em multi-image variety} 
$M(C_1,\ldots,C_n)$ to be the closure of the image of  (\ref{eq:mapCCC}).
This is an irreducible subvariety in the product of Grassmannians
$\,{\rm Gr}(1,\PP^3)^n \subset (\PP^5)^n$.
We expect the map (\ref{eq:mapCCC}) to be birational in almost all cases,
so $M(C_1,\ldots,C_n)$ is a threefold.

The multi-image variety is clearly contained in the concurrent lines variety. In symbols,
\begin{equation}
\label{eq:MCCC}
M(C_1,\ldots,C_n) \,\subseteq \,
V_n \,\cap\, (C_1 \times \cdots \times C_n) 
\,\,\subset \,\,{\rm Gr}(1,\PP^3)^n.
\end{equation}
Our first result in this section shows that  the left inclusion in (\ref{eq:MCCC}) is usually an equality.

\begin{theorem}
\label{thm:pairwisedisjoint}
Suppose that the $n$ focal loci $\mathcal{F}(C_1),\ldots,\mathcal{F}(C_n)$
are pairwise disjoint. Then
\begin{equation}
\label{eq:MequalsV}
M(C_1,\ldots,C_n) \,= \,
V_n \,\cap\, (C_1 \times \cdots \times C_n),
\end{equation}
i.e.,~the concurrent lines variety gives an implicit representation of the multi-image variety.
\end{theorem}

\begin{proof}
By (\ref{eq:MCCC}), we only need to show one direction.
For $(L_1,\ldots,L_n) \in  V_n \cap M(C_1,\ldots,C_n)$,
there exists $x \in \PP^3$ such that
$x \in L_i$ for all $i$. If $x$ does not lie in any of the $n$ focal loci 
then $L_i = C_i(x)$ and we are done. Otherwise, $x$ lies in
exactly one of the focal loci, say, $x \in \mathcal{F}(C_i)$.
We then consider a sequence of nearby pairs $(x_\epsilon,L_{i,\epsilon})$ 
that converges to $(x,L_i)$ and satisfies
 $x_\epsilon \in L_{i,\epsilon} \backslash \mathcal{F}(C_i)$
 and $C_i(x_\epsilon) = L_{i,\epsilon}$.
 For each $j \in \{1,2,\ldots,n\} \backslash \{i\}$ the locus
 $\mathcal{F}(C_j)$ is closed. Since it does not contain $x$,
 we can assume that it also
   does not contain $x_\epsilon$. Hence
$ (C_1(x_\epsilon),\ldots,C_n(x_\epsilon))$ is a well-defined
sequence of points in the variety $M(C_1,\ldots,C_n)$. It converges
to $(L_1,\ldots,L_n)$, which therefore also lies in $M(C_1,\ldots,C_n)$.
\end{proof}

We next undertake
a detailed study of two special cases.
Subsection \ref{subsec51} concerns arbitrary $n$, but $\beta_i \in \{0,1\}$.
In Subsection \ref{subsec52} we focus on 
$n=2$, but with arbitrary $\beta_1$ and $ \beta_2$.

\subsection{Multiple Views with Pinhole and Two-Slit Cameras}
\label{subsec51}

We begin with an example for $n=3$ that shows the necessity of the
 hypothesis on the focal loci in
Theorem \ref{thm:pairwisedisjoint}.
The concurrent lines variety $V_3$ is a $9$-dimensional
subvariety of $\PP^5 \times \PP^5 \times \PP^5$.
Its ideal $I_3$ is generated by six quadrics and ten cubics in
$\RR[p_{01},\ldots,p_{23},q_{01},\ldots,q_{23},r_{01},\ldots,r_{23}]$.
Given three congruences $C_1,C_2,C_3$, we are interested in the
variety  $  (C_1 \times C_2 \times C_3) \cap V_3$.
This contains the threefold $M(C_1,C_2,C_3)$, possibly strictly.

\begin{example} \label{ex:sixlines} \rm
Let $\beta_1=\beta_2 = \beta_3 = 1$ and fix the two-slit cameras $C_1,C_2,C_3$
defined by
$$ J \,\,=\,\, \langle p_{01},p_{23}, q_{02},q_{13}, r_{03},r_{12} \rangle \quad
\subset \quad
\RR[p_{01},\ldots,p_{23},q_{01},\ldots,q_{23},r_{01},\ldots,r_{23}].
$$
Geometrically, we partition the set of six coordinate lines in $\PP^3$
into three pairs of disjoint lines. Each pair defines a $(1,1)$-congruence.
Note that $\mathcal{F}(C_1)$, $ \mathcal{F}(C_2)$ and $ \mathcal{F}(C_3)$ are
distinct, but they intersect in the four coordinate points. So, the hypothesis of
Theorem \ref{thm:pairwisedisjoint} fails.

The ideal $J+I_3 $ is radical but not prime. It is the intersection of
five primes, each defining a threefold in $\PP^5 {\times} \PP^5 {\times} \PP^5$.
One of these is the toric variety $M(C_1,C_2,C_3)$, with ideal
$$ \begin{matrix}
\langle
  p_{03} p_{12}-p_{02} p_{13}, q_{03} q_{12}+q_{01} q_{23}, r_{02} r_{13}-r_{01} r_{23},
   p_{03}    q_{12} + p_{12} q_{03} , p_{13} r_{02}+p_{02} r_{13}, q_{23} r_{01}+q_{01} r_{23}, \\
      p_{12} q_{23} r_{13}+p_{13} q_{12} r_{23},  \, p_{02} q_{23} r_{13}+p_{03} q_{12} r_{23},\,
       p_{03} q_{12} r_{13}+p_{13} q_{01} r_{23}, \,
      p_{02} q_{12} r_{13}+p_{12} q_{01} r_{23},  \\
\!\!\!   \!\!\!         p_{02} q_{03} r_{13}-p_{03} q_{01} r_{23},\,\,\,
             p_{03} q_{23} r_{02} +p_{02} q_{03} r_{23},\,\,\,
      p_{03} q_{12} r_{02} - p_{02} q_{01}  r_{23},   \\
\qquad       p_{03} q_{12} r_{01} - p_{02} q_{01}  r_{13},\,\,\,
       p_{02} q_{12} r_{01} + p_{12} q_{01} r_{02},\,\,\,
      p_{13} q_{03} r_{01} + p_{03} q_{01} r_{13}
               \,\rangle \,\,+\,\, J.
 \end{matrix}
$$
The other four associated primes define coordinate $3$-planes in $\PP^5 \times \PP^5 \times \PP^5$.
They are
$$ \begin{matrix} 
\langle p_{12},p_{13}, q_{12},q_{23}, r_{13}, r_{23} \rangle + J , &
 \langle p_{02},p_{03}, q_{03},q_{23}, r_{02},r_{23} \rangle + J, \\
 \langle p_{03},p_{13}, q_{01},q_{03}, r_{01},r_{13} \rangle + J, &
 \langle p_{02},p_{12}, q_{01},q_{12}, r_{01},r_{02} \rangle + J. 
 \end{matrix}
$$ 
To understand the geometric meaning of these extraneous components,
consider the last ideal. It represents all
triples $(L_1,L_2,L_3)$ where
$L_1,L_2,L_3$ pass through $(0{:}0{:}0{:} 1)$,
and each line $L_i$ intersects one of  the opposite coordinate lines,
as is required for lines in $C_i$.
\hfill $\diamondsuit$
\end{example}

From now on we consider congruences whose focal loci are pairwise disjoint,
so the identity~(\ref{eq:MequalsV}) holds.
We begin with the most classical case, where $C_1, \ldots, C_n$ are pinhole cameras with distinct centers $c_1,\ldots, c_n$.
Each congruence $C_i$ is a plane $\PP^2$ in $\PP^5$, and the
map $x \mapsto C_i(x) = x \vee c_i $ is analogous to  the linear projection 
$\PP^3 \dashrightarrow \PP^2$ with center $c_i$.
In the usual set-up of photographic cameras \cite{AST, THP},  this map is
represented by a $3 \times 4$-matrix $A_i$ whose kernel is given by
$c_i$, and $\PP^2$ is identified with the image of $A_i$.
Since $A_i$ and $x \mapsto x \vee c_i$ have the same kernel, 
there exists a $6 \times 3$-matrix $B_i$ such that $x \vee c_i = B_i A_i x $; see also Section~\ref{sec7}.

\begin{proposition} \label{prop:multiview}
The multi-view variety (in the sense of \cite{AST,THP})
of the photographic cameras $A_1,\ldots,A_n$
 is isomorphic to the multi-image variety
$M(C_1,\ldots,C_n)$ under the map
$$ (\PP^2)^n \rightarrow {\rm Gr}(1,\PP^3)^n,\,\,
(u_1, \ldots, u_n) \mapsto (B_1 u_1, \ldots, B_n u_n) . $$
Here, the equation  (\ref{eq:MequalsV}) holds ideal-theoretically,
i.e.,~the prime ideal of the multi-view variety is the image of $I_n$
modulo the linear equations $P_1 c_1 = \cdots = P_n c_n = 0$ that define $C_1 \times \cdots \times C_n$.
\end{proposition}

\begin{proof} The first statement is immediate from the discussion of the
two realizations of $\PP^2$, as the image of $A_i$ or as the plane $C_i$ in $\PP^5$.
The second statement about ideals is more subtle. It can be derived using the
functorial set-up developed by Li  \cite{Li}. Both schemes represent the same
functor, so they are isomorphic as in \cite[Proposition 2.8]{Li}. The isomorphism is
compatible with the initial degeneration in \cite[\S 3]{Li}. The identification of
ideals follows.
\end{proof}

The ideal $I_n$ of the concurrent lines variety $V_n$ is
minimally generated by $\binom{n+1}{2}$ quadrics and $10 \binom{n}{3}$ cubics
in the $6n$ Pl\"ucker coordinates. We add to this the $3n$ linear equations
that define $C_1 \times \cdots \times C_n$. This reduces the minimal generators
to $\binom{n}{2}$ quadrics and $\binom{n}{3}$ cubics. These are the
 bilinearities and trilinearities, well-known in the
 computer vision community \cite{HZ, THP},  that link two and three views.
For an algebraic derivation see \cite[Corollary 2.7]{AST}.

\begin{example} \rm
Let $n=4$ and take $c_1,c_2,c_3,c_4$ to be the four coordinate points in $\PP^3$.
Using notation as in Example \ref{ex:sixlines}, we represent
$C_1 \times C_2 \times C_3 \times C_4$ by the prime ideal
$$ J \,\, = \,\, \langle 
\, p_{12}, p_{13}, p_{23},\,
\, q_{02}, q_{03}, q_{23},\,
\, r_{01}, r_{03}, r_{13},\,
\, s_{01}, s_{02}, s_{12}
\, \rangle.
$$
The concurrent lines ideal $I_4$ is generated by $10$ quadrics and $40$ cubics.
Their sum $I_4+J$ is a prime ideal. Modulo $J$, it is generated by
$6$ quadratic binomials and $4$ cubic binomials. As in \cite[Proposition 4.1]{AST}, these
are the relations among the off-diagonal entries of $4 \times 4$-matrices
$$ \begin{pmatrix}
* & p_{01}  & p_{02} &  p_{03}\, \\
- q_{01} & * & q_{12} & q_{13} \,\\
- r_{02} & -r_{12} & * & r_{23} \,\\
-s_{03} & -s_{13} & -s_{23} & *\, 
\end{pmatrix}
$$
that  have rank $1$. For instance 
$q_{12} s_{03}+q_{01} s_{23}$ and
$p_{01} q_{12} r_{02}+p_{02} q_{01} r_{12} $ are in $I_4+J$.
\hfill $\diamondsuit$
\end{example}

We next generalize Proposition \ref{prop:multiview} to arrangements
of $n_1$ pinhole cameras $C_1,\ldots,C_{n_1}$
 and $n_2$ two-slit cameras $C_1',\ldots,C_{n_2}'$.  These
$n = n_1 + n_2$ cameras are assumed to satisfy the hypothesis 
of Theorem~\ref{thm:pairwisedisjoint}. Thus, the
pinholes are distinct, the slits are pairwise disjoint, and no pinhole  is
allowed to lie on a slit. The following is our main result in this section:

\begin{theorem}\label{th:multi-image_linear}
The ideal of the multi-image variety $ M(C_1,\ldots,C_{n_1},C'_1,\ldots,C'_{n_2})$
 is minimally generated by 
$3 n_1+2 n_2$ linear forms, $\binom{n_1+n_2}{2} + n_2$ quadrics, and
$\binom{n_1}{3}+3 \binom{n_1}{2} n_2+ 6 n_1 \binom{n_2}{2} + 10 \binom{n_2}{3}$  cubics
in the $6n_1 + 6n_2$ Pl\"ucker coordinates on the ambient space $(\PP^5)^{n_1+n_2}$.
\end{theorem}

Note that for $n_2 = 0$ we recover the known ideal generators of the multi-view variety \cite{AST}.

\begin{proof}
The desired ideal is obtained from $I_n$ by adding $3$ linear forms for every pinhole camera $C_i$
and $2$ linear forms for every two-slit camera $C_i'$. We need to examine the extent
to which  the generators of $I_n$  become linearly dependent modulo these
$3 n_1 + 2n_2$ linear forms. For $n \leq 3$ cameras this examination amounts to 
computations with {\tt Macaulay2}, one for each ordered partition  $(n_1,n_2)$ of $n$.
 For $n \geq 4$  cameras we group
the minimal generators of $I_n$ according to their degree in the $\ZZ^n$-grading.
Each graded component specifies a subset of cameras of size at most three.
Hence all the linear relations arise from those for $n=3$.
\end{proof}

\subsection{Epipolar Geometry for Rational Cameras}
\label{subsec52}

In this subsection we take a closer look at the case of
two rational cameras $C_1$ and $C_2$. We assume
that $C_i$ is a congruence of bidegree $(1,\beta_i)$
for $i=1,2$ and that $\mathcal{F}(C_1) \cap \mathcal{F}(C_2) = \emptyset$.
 The associated multi-image variety $M(C_1, C_2)$ in  $\PP^5 \times \PP^5$
 is defined by the ideal
 \begin{equation}
 \label{eq:ideal52}
 I_P(C_1) \,+\, I_Q(C_2) \,+\, \langle {\rm trace}(PQ^*)\rangle
\quad\subset \,\,\,\CC[p_{01},\ldots,p_{23},q_{01},\ldots,q_{23}],
\end{equation}
 where $I_P(C_1)$ and $I_Q(C_2)$ are respectively the ideals for $C_1$ and $C_2$ in the two sets of variables.

This set-up generalizes familiar objects from two-view geometry.
For example, if $p$ is a line in the congruence $C_1$, then the
{\em epipolar curve} ${\rm Epi}(p)$ in $C_2$ consists of all lines $q$ such that $(p,q)$ belongs to $M(C_1,C_2)$. The ideal of ${\rm Epi}(p)$ in the $\PP^5$ with coordinates $q_{01},\ldots,q_{23}$ is given by $ I_Q(C_2) + \langle {\rm trace}(PQ^*) \rangle$. The curve ${\rm Epi}(p)$ has degree $1+\beta_2$ in Pl\"ucker coordinates (see Proposition \ref{prop:epipolar} below). In particular, for pinhole cameras $C_1, C_2$, we recover 
the classical epipolar lines in two-view geometry \cite{HZ}. However, if either $C_1$ or $C_2$ is not a pinhole cameras, then the families of curves ${\rm Epi}_{12}=\{{\rm Epi}(p) : p \in C_1\}$ and ${\rm Epi}_{21}=\{{\rm Epi}(q) : q \in C_2\}$ are not related by a one-to-one correspondence. More concretely: if $q$ and $q'$ both belong to ${\rm Epi}(p)$, then we cannot conclude that ${\rm Epi}(q)={\rm Epi}(q')$. This follows from the fact that the ideal from Theorem~\ref{th:multi-image_linear} is not multilinear. This contrasts with the classical case, where there exists a homography relating the epipolar lines in each image, which are isomorphic to a $\PP^1$.

In traditional two-view geometry, the two camera centers in $\PP^3$
span the ``baseline'', which projects onto the two epipoles. This generalizes as follows to our setting.
A line $L$ in $\PP^3$ is  a {\em baseline} for the two cameras $C_1$ and $C_2$
if it lies in the intersection $C_1 \cap C_2$ in ${\rm Gr}(1,\PP^3)$. 
The baselines are precisely the loci that are contracted by the map~\eqref{eq:mapCCC}, since for every point $x$ in such a line $L$ we have
$(C_1(x),C_2(x)) = (L,L)$.
We expect $C_1 \cap C_2$ to consist of finitely many points.
Some of these points are defined over $\CC$.
These are included in our count.

\begin{proposition}\label{prop:epipolar}
Let $C_1$ and $C_2$ be general congruences of bidegree $(1,\beta_1)$ and
$(1,\beta_2)$. The epipolar curves ${\rm Epi}(p)$ and ${\rm Epi}(q)$ 
in ${\rm Gr}(1,\PP^3)$ have degrees $1+\beta_2$ and $1  + \beta_1$ respectively.
The number of baselines in $\PP^3$ for the camera pair $(C_1,C_2)$  equals
$ 1 + \beta_1 \beta_2$.
\end{proposition}

\begin{proof}
The intersection theory in the Grassmannian ${\rm Gr}(1,\PP^3) \subset \PP^5$ works as follows.
 A hypersurface of degree $d$
intersects an $(\alpha,\beta)$-congruence in a curve of degree $d \alpha + d \beta$.
Two congruences of bidegrees $(\alpha_1,\beta_1)$ and
$(\alpha_2,\beta_2) $ intersect in $\alpha_1 \alpha_2 + \beta_1 \beta_2$ points.
A classical reference is Jessop's book \cite{Jes}. A modern one is any
introduction to {\em Schubert calculus}.
\end{proof}

\noindent We now illustrate the concepts introduced in this subsection
with an example.

\begin{example} \rm

Let $\beta_1 = \beta_2=2$ and consider the type 3 congruences $C_1$ and $C_2$ of common transversals to $L_1, X_1$ and $L_2, X_2$, where
\[
L_1=V(x_1,x_2-x_3), \,\,\,\, X_1=V(x_0, x_1^2+x_2^2-x_3^2 ),
\]
\[
L_2=V(x_0 - x_1 , x_2),  \,\,\,\, X_2=V(x_0^2-x_1^2+x_2^2,x_3).
\] 
Note that that the intersection of $ \mathcal{F}(C_1) = L_1 \cup X_1$
and $ \mathcal{F}(C_2) \,=\, L_2 \cup X_2 $ in $\PP^3$ is empty.
The intersection points on the two focal loci are $L_1 \cap X_1 = \{(0{:}0{:}1{:}1)\} $
and $L_2 \cap X_2 = \{(1{:}1{:}0{:}0)\}$.
The prime ideals of the two congruences in the coordinate ring of ${\rm Gr}(1,\PP^3)$ are given by
\[
\label{eq:C1}
I_P(C_1) \, = \, \langle p_{12}-p_{13}, \,\, p_{01}^2+p_{02}^2
       -p_{03}^2, \,\, p_{01}p_{13}+p_{02}p_{23}+p_{03}p_{23}, \,\, p_{01} p_{23}-p_{02} p_{13}+p_{03} p_{12} \rangle,
\]
\[
\label{eq:C2}
 I_P(C_2) \, = \, \  \langle p_{02}-p_{12} , \,\, p_{03}^2-p_{13}^2+p_{23}^2, \,\, p_{01}p_{03}+p_{01}p_{13}-p_{12}p_{23}, \,\, p_{01} p_{23}-p_{02} p_{13}+p_{03} p_{12} \rangle .
\]
In both expressions, the first two polynomials are the Chow forms of $L_i$ and  $X_i$ respectively.
The ideal of the two-image variety $M(C_1, C_2)$ is given by \eqref{eq:ideal52}.

If we fix a point $p$ in $C_1$ then its corresponding cubic curve  ${\rm Epi}(p)$ lives in $C_2$,
 and vice versa. For example, the ideal $I_P(C_2)+\langle  3p_{01} + 4p_{02} + 5p_{03} - 3p_{12} - 3p_{13} + p_{23} \rangle$ defines the epipolar curve in ${\rm Gr}(1,\PP^3)$ associated with $p=( 3 : 4 : 5 : -3 : -3 : 1)$ in $C_1$.

The ideal $I_P(C_1)+I_P(C_2)$ defines five points in ${\rm Gr}(1,\PP^3)$,
These represent the five baselines.
One point is $(0:1:1:1:1:0)$. It represents the line 
through $(0{:}0{:}1{:}1)$ and $(1{:}1{:}0{:}0)$.
 The other four baselines have the Pl\"ucker vectors
 $$ \qquad \biggl(\,
\frac{5}{2}a^2-\frac{1}{2}\,  : \,a\, : \,-\frac{5}{2}a^3 - \frac{1}{2}a \,:\, a\, :\,a\,  :\, 1
 \biggr)  \qquad \hbox{where} \quad  5 a^4-2a^2+1 = 0. 
 $$
 We see that three of the five baselines are real.
 The other two are defined over $\CC$.
   \hfill $\diamondsuit$
\end{example}

\section{Higher-Order Cameras}
\label{sec6}

In Sections \ref{sec4} and \ref{sec5} we considered congruences whose
point-to-line maps $x \mapsto C(x)$ are rational.  However,
researchers in computer vision have also studied non-standard cameras
that are algebraic of higher order (see e.g.~\cite{SRT}). For example,
a $(2,\beta)$-congruence associates a given point $x$ with a pair of
lines, but the corresponding physical camera might record only one
line for $x$, due to orientation constraints.  Using higher order
cameras also allows the possibility of triangulating the position of
3D points from a single picture (if the camera is known, as in
\cite{HKS}). In this section we develop algebraic geometry for two
types of devices that exist in practice, namely non-central panoramic
cameras and catadioptric cameras.
 
\subsection{Panoramic Cameras}

A panoramic camera enables photographs with a $360^\circ$ field of view.
One such panoramic device consists of a 1D-sensor measuring 2D-projections onto a fixed center, that is rotated about a vertical axis $L$ not containing the center. The 1D-sensor travels
on a circle $X$ around the line $L$.
The associated congruence $C$ consists of all lines that intersect both $L$ and $X$.
This  has bidegree $(2,2)$.
Physical realizations come in two versions. Depending on 
the orientation of the sensor, precisely one of the
two lines of $C$ through a point $x$ is being recorded.
If the sensor points outwards then we get a {\em non-central panoramic camera}.
This is shown on the left in 
Figure~\ref{fig:panorama}. 
If the sensor points inwards then the camera is a {\em cyclograph}, a device that records a $360^\circ$ representation of a single  object placed in the middle.  

Another system is the {\em stereo panoramic camera}, on the right in  Figure~\ref{fig:panorama}.
It is obtained by rotating a 1D-sensor about an axis parallel to the sensor. In each position the sensor records parallel lines tangent to the rotation.
 This is a variation of the camera proposed in~\cite{HKS}. It produces stereo (binocular) panoramic images, since every 3D-point is observed from two sensor locations.
 The paper \cite{SRT} features
both of the cameras shown in Figure~\ref{fig:panorama}.

\begin{figure}[h]
\centering
\includegraphics[height=4cm]{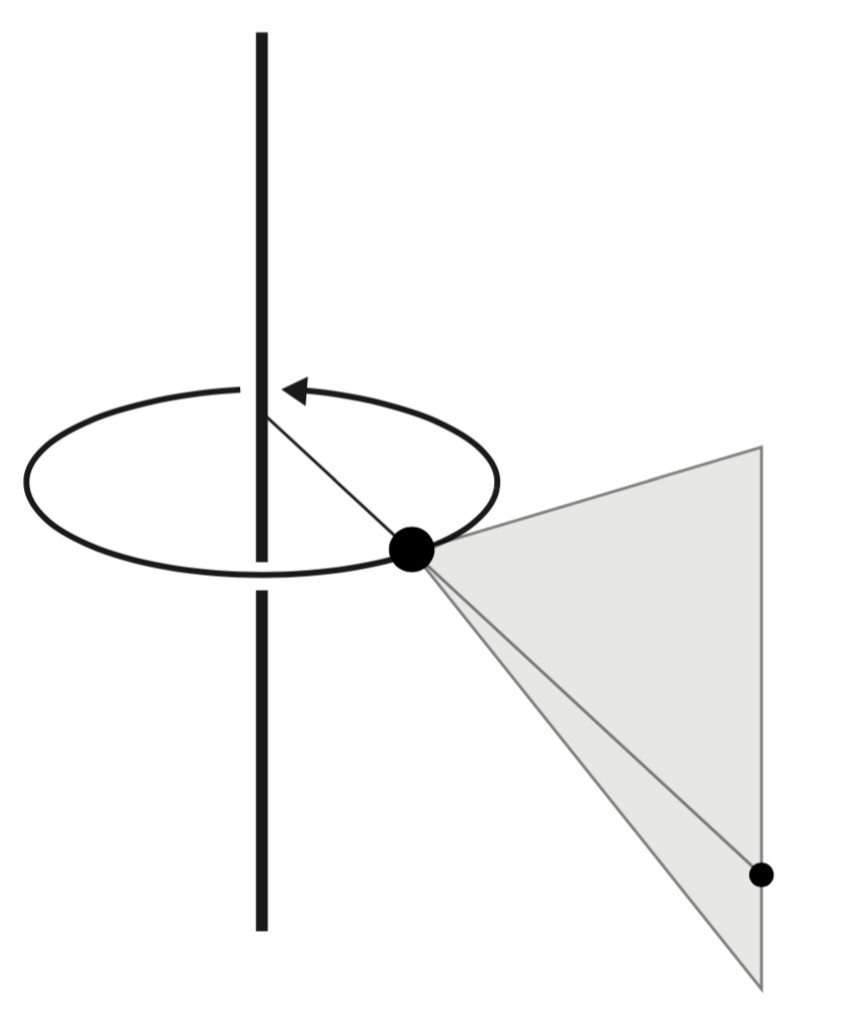}\qquad
\includegraphics[height=4.2cm]{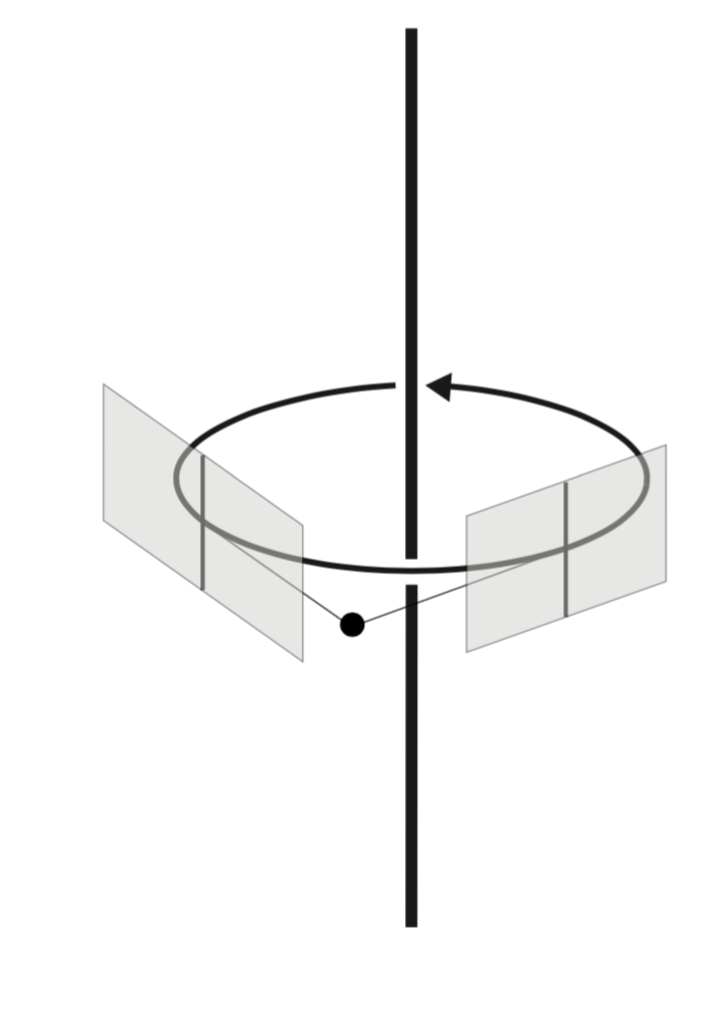}
\caption{Panoramic cameras:  non-central ({\em left}) and stereo ({\em right})}
\label{fig:panorama}
\end{figure}

We first discuss the non-central panoramic camera. The  corresponding
$(2,2)$-congruence $C$ is determined by a line $L$ and a non-degenerate conic $X$, both in $\PP^3$, such that $L \cap X = \emptyset$.
Then, as above, $C$ consists of all lines in $\PP^3$ that intersect both $L$ and $X$.

\begin{lemma} \label{lem:conicline}
Any two such congruences are equivalent up to projective transformations of~$\PP^3$.
\end{lemma}
\begin{proof} Given any two pairs of disjoint conics and lines $(X, L)$ and $(X', L')$ in $\PP^3$, we may always apply a homography over $\CC$ so that $X=X'=V(x_0,x_1^2+x_2^2+x_3^2)$. Transformations that fix $X$ are projectivizations of affine maps $\CC^3 \rightarrow \CC^3,\, \tilde x \mapsto A\tilde x + b$ such that $AA^T = \lambda {\rm Id}$, where $\tilde x = (x_1/x_0,x_2/x_0,x_3/x_0)$ are coordinates on the affine chart $U=\{x_0 \ne 0 \}$. These maps act transitively on points of $U_0$ and on points of $V(x_0) \backslash X$,
 so we conclude that $L$ and $L'$ are equivalent. If we restrict to $\RR$, 
 and both conics have real points, then we use $X=X'=V(x_0,x_1^2+x_2^2-x_3^2)$, and a similar result holds.
\end{proof}

Thanks to Lemma~\ref{lem:conicline}, we may choose 
$\,I_L \,= \,\langle x_1, x_2 \rangle \,$ and $\,
I_X \,=\, \langle x_3,x_1^2+x_2^2-x_0^2 \rangle \,$
as  the ideals that represent the line $L$ and the conic $X$.
The ideal of the congruence $C$ in $\RR[p_{01},\ldots,p_{23}]$
is generated by the Chow forms of $L$ and $X$ together with  the Pl\"ucker quadric:
\begin{equation}
\label{eq:specialkummer}
I_C \, = \,\langle \,p_{12}\, ,  \,\, p_{03}^2-p_{13}^2-p_{23}^2 \, , \,\,
p_{01} p_{23} - p_{02} p_{13} + p_{03} p_{12}  \rangle.
\end{equation}
We can see that $C$ has bidegree $ (2,2) $ by intersecting it with
generic $\alpha$-planes and $\beta$-planes. Indeed,
for generic vectors $u,v \in \RR^4$, we verify
$\, {\rm degree}(I_C + \langle P u \rangle) = 
{\rm degree}(I_C + \langle P^* v \rangle) = 2  $.

\begin{proposition} \label{prop:non-central}
The focal locus of the non-central panoramic camera 
 consists of the plane spanned by the conic $X$, taken with multiplicity $2$,
and a conjugate pair of complex planes that intersect in the line $L$.
Algebraically, it is defined by the non-reduced quartic $(x_1^2+x_2^2) x_3^2$.
\end{proposition}

\begin{proof}
We compute the focal locus as described at the end of Section \ref{sec2}.
For a generic $(2,2)$-congruence, this yields a quartic polynomial
defining a Kummer surface as in~\cite{Kum}.
For the special congruence $C$ given in (\ref{eq:specialkummer}),
the quartic  generator is found to be
$(x_1^2+x_2^2) x_3^2$.
\end{proof}

\begin{remark}  \label{rem:panoramic} \rm
A smooth $(2,2)$-congruence is a del Pezzo surface 
of degree $4$. Its $16$ straight lines 
correspond to a matching between the
$16$ singular points and the $16$ special planes
of its Kummer surface (cf.~Example  \ref{ex:twotwo} and \cite{Kum}).
The congruence $C$ in (\ref{eq:specialkummer}) is singular.
Its singular locus, $V(p_{03},p_{12},p_{13},p_{23})$, consists of
all lines in $\PP^3$ that meet $L$ and lie in the plane spanned by $X$.
The Kummer surface degenerates to the arrangement  of four planes $\mathcal{F}(C)$
in Proposition \ref{prop:non-central}.
It would be desirable to get a better understanding 
of such degenerations~of $(2,2)$-congruences.
One approach to this is sketched by Kummer in
\cite[\S XXXVII, page 71]{Kum}.
\end{remark}

We now discuss the stereo panoramic camera. Its congruence
consists of the lines that are tangent to a singular quadratic surface $Q$
and pass through a fixed line $L$. The pair $(Q,L)$ is unique up to projective
transformations of $\PP^3$. For the camera on the right in
Figure \ref{fig:panorama}, the quadric $Q$ is a cylinder around the axis 
and $L$  is a line at infinity. We note that the stereo panoramic camera
is dual, in the sense of projective geometry, to the non-central panoramic camera. 
Specifically, its congruence $C^*$ is obtained by dualizing $C$ in (\ref{eq:specialkummer}). The result is
\begin{equation}
\label{eq:specialkummer2}
I_{C^*} \, = \,\langle \,p_{03}\, ,  \,\, p_{12}^2-p_{02}^2-p_{01}^2 \, , \,\,
p_{01} p_{23} - p_{02} p_{13} + p_{03} p_{12}  \rangle.
\end{equation}
Here $I_L = \langle x_0, x_3 \rangle$ and $I_Q = \langle x_0^2-x_1^2-x_2^2 \rangle$.
Clearly, $C^*$ has bidegree $(2,2)$.
As in Remark~\ref{rem:panoramic}, $C^*$ is singular along a line.
Singular points are  lines that meet $L$ and the cone point of $Q$.

\begin{corollary} \label{cor:stereo}
The focal locus $\mathcal{F}(C^*)$ of the stereo panoramic camera 
consists of the singular quadric $Q$ and  the plane at infinity (spanned by $L$ and the cone point of $Q$),
which is taken with multiplicity $2$.
Algebraically, it is defined by the non-reduced quartic $x_0^2 (x_0^2 - x_1^2 - x_2^2)$.
\end{corollary}

\begin{proof}
This is verified by a  computation, like Proposition \ref{prop:non-central}.
\end{proof}

\subsection{Catadioptric Cameras}

A {\em catadioptric camera} is an optical system that makes use of reflective surfaces (catoptrics) and lenses (dioptrics). Mirrors can bring several advantages, such as a larger field of view or better focusing properties. For many applications it is desirable to have a single effective viewpoint \cite{BN}.
This is usually achieved by using a mirror that is a paraboloid or hyperboloid of revolution, placing a pinhole camera at one of the foci.
In our setting, it is natural to consider a catadioptric system
that uses an arbitrary smooth algebraic surface $S$ as a mirror, and a geometric camera $C$ to record lines. 
We shall describe the resulting line congruence.

We measure angles in $\PP^3$ according to the usual scalar product in $U_0=\{x_0 \ne 0\}$, so that
\begin{equation}
\cos \,\angle(x,y) \,\,\,= \,\,\,\frac{x_1 y_1 + x_2 y_2 + x_3 y_3}{\sqrt{(x_1^2+x_2^2+x_3^2)(y_1^2+y_2^2+y_3^2)}}.
\end{equation}
The points $(0 : x_1 : x_2 : x_3)$ and $(0 : y_1 : y_2 : y_3)$
lie on  the plane at infinity, $\PP^3  \setminus U_0$, and they
represent directions in $U_0$. Let $H$ be the plane in $\PP^3$
defined by $a_0 x_0 + a_1 x_1 + a_2 x_2 + a_3 x_3=0$.
Here $a_i \in \CC$ is allowed, but
we assume that $H$ is {\em non-isotropic}, meaning that $a_1^2 + a_2^2 + a_3^2 \not= 0$.
With the convention above, the {\em reflection of a point} $(y_0 : y_1 : y_2 : y_3)$ with 
respect to $H$ is 
\begin{equation}\label{eq:point_reflection}
\rho_H(y)=\begin{bmatrix}
\left (\sum_{i=1}^3 a_i^2\right ) y_0  \smallskip \\
  \left (\sum_{i=1}^3 a_i^2 \right) y_1 - 2\left (\sum_{i=0}^3 a_i y_i\right ) a_1 \smallskip \\
  \left (\sum_{i=1}^3 a_i^2 \right) y_2 - 2\left (\sum_{i=0}^3 a_i y_i\right ) a_2 \smallskip  \\
 \left (\sum_{i=1}^3 a_i^2\right) y_3 - 2\left (\sum_{i=0}^3 a_i y_i\right ) a_3
 \end{bmatrix}.
\end{equation}
This map is a linear involution of $\PP^3$ that fixes $H$.
Its restriction to the real affine $3$-space $U_0$ is the usual Euclidean reflection with respect to $H \cap U_0$.
The reflection of a line $p=x\vee y$ with respect to $H$ is defined  as $\rho_H(x) \vee \rho_H(y)$. This map is an involution of the Grassmannian ${\rm Gr}(1,\PP^3)$.  It acts on Pl\"ucker coordinates
by a linear involution  on the ambient $\PP^5$.
The  $6 \times 6$-matrix defining that involution is the
second compound matrix $\wedge_2 \rho_H$ of the $4 \times 4$-matrix~$\rho_H$.

Let $S$ be a smooth algebraic surface in $\PP^3$ defined by a polynomial $f$ of degree $d$.
Two lines $L$ and $L'$ in $\PP^3$ are 
{\em specular} for $S$ if there exists a point $x \in S$ such that the tangent plane $T_x S$ is not isotropic, 
$L$ and $L'$ meet in $x$, and they are reflections of each other respect to $T_x S$. 
We define the  {\em mirror variety} $M_S$ to be the closure of the set of all pairs $(L,L') \in {\rm Gr}(1,\PP^3)^2$
that are specular for $S$. For a general line $L$ there are $d$ lines $L'$
such that $(L,L') \in M_S$, one for each point $x $ in $S \cap L$. Hence
the mirror variety $M_S$ of a surface $S$ is $4$-dimensional.

To compute the defining equations of the mirror variety $M_S$, we first construct the ideal 
\[J\,\,= \big\langle \,f,\, P x,  {\rm trace}(PP^*), {\rm trace}(QQ^*) \big\rangle \,+\,
\big\langle \wedge_2  \bigl( q \,|\, \rho_{T_x S} (p) \bigr) \big\rangle.
\] 
This lives in $\RR[x_0,x_1,x_2,x_2,p_{01},\ldots,p_{23},q_{01},\ldots, q_{23}]$.
The last summand is the ideal of $2 \times 2$-minors of a $6 \times 2$-matrix,
where the second column is the reflection of the line $p$ with respect to the tangent plane $T_x S$. 
It expresses the requirement that $q$ is equal to that reflection.
We then saturate $J$ with respect to the isotropic ideal $\,I_{\rm Iso}=\langle\, (\nabla_x f)_1^2 + (\nabla_x f)_2^2
 + (\nabla_x f)_3^2 \,\rangle$ and with respect to the irrelevant ideal $\langle x_0,x_1,x_2,x_3 \rangle$, 
 before eliminating the variables $x_0, x_1,x_2, x_3$.

\begin{example} \label{ex:ellipsoid}
 \rm Let $S$ be the ellipsoid given by 
$f = \frac 1 {16} (x_1^2+ x_2^2)+\frac 1 {25} x_3^2- x_0^2$. The mirror~variety $M_{S}$ has codimension $6$ in $\PP^5 \times \PP^5$, and bidegree $4 t_0^5 t_1+12 t_0^4 t_1^2+18 t_0^3 t_1^3+12 t_0^2 t _1^4+4 t_0 t_1^5$. 
\hfill $\diamondsuit$
\end{example}

\begin{remark} \rm The intersection of the mirror variety $M_S$ with the diagonal $\Delta$ in $\PP^5 \times \PP^5$ is the {\em normal congruence}. These are the
 lines that intersect $S$ orthogonally (we assume that we have removed components associated with isotropic tangent planes). The focal locus of the normal congruence is the {\em caustic surface} \cite{CT}. In the 
 language of differential geometry, this is the union of the centers of principal curvature for $S$. 
 The order $\alpha$ of the normal congruence coincides with the
   {\em Euclidean distance degree} (ED degree) of the surface $S$.
That is the number of critical points on $S$ of the squared distance function to a generic point \cite{DH}.
\end{remark}

Let $C$ be any congruence, representing a geometric camera.  The {\em
  specular congruence} of $C$ with respect to $S$ is another surface
$C_S$ in the Grassmannian ${\rm Gr}(1,\PP^3)$. We define $C_S$ as the
closure of the set of all lines $L' $ for which there exist $L \in C$
and $x \in L \cap S$ such that $T_x S$ is not isotropic and $L' =
\rho_{T_x S} L$.  Concretely, $C_S$ contains the lines of $C$ after
these are reflected by $S$.  Thus, $C_S$ is the congruence associated
with the catadioptric camera determined by $S$ and $C$.  Note that if
$L'$ is in $C$ then there exists $L$ such that $(L,L') \in M_S$. This
implies
\begin{equation}\label{eq:specular_section}
C_S \, \subseteq  \, \pi_2 \bigl( M_S \, \cap \,(C \times {\rm Gr}(1,\PP^3))\bigr) \,=: \, C_S',
\end{equation}
where $\pi_2$ is the projection onto the second factor.
For a general $C$, an appropriate application of Bertini's Theorem ensures that
the right hand side $C_S'$ is irreducible, and the containment \eqref{eq:specular_section} is an equality
(set-theoretically). In this case, we can compute equations for $C_S$ by adding the equations defining
the given congruence $C$ (in the variables $p_{01},\ldots,p_{23}$) to the ideal of mirror variety $M_S$, 
then saturating by the irrelevant ideal $\langle p_{01},\ldots, p_{23} \rangle$, and finally 
eliminating the variables $p_{01},\ldots, p_{23}$. We experimented with this in {\tt Macaulay2}.

The next example shows that  $C_S'$ can have spurious components. These are removed by saturating the ideal of $C'_S$ with respect to the Chow form $Ch_{X_{\rm Iso}}$ where 
$X_{\rm Iso}=V(I_{\rm Iso} + \langle f \rangle)$. 
We note that the order and class of the specular congruence $C_S$ depend on the relative position of $S$ and $C$ (and the absolute quadric $V(x_0, x_1^2+x_2^2+x_3^2)$). The focal locus of $C_S$ is a 
{\em caustic by reflection} in~\cite{JP}, but here we do not require for the light source to be a point.

\begin{example} \rm Let $S$ be the ellipsoid from Example \ref{ex:ellipsoid}.
 We first consider a catadioptric camera with mirror $S$ and a pinhole sensor at  
 a point $P$. Let us start with $P=(1:0:0:3)$. 
 The radical ideal of $C_S'$ (computed as described above) is the intersection of two prime ideals:
\[
\begin{aligned}
&I_1=\langle q_{12},3 q_{02}-q_{23},3 q_{01}-q_{13}\rangle \\
&I_2=\langle q_{03} q_{12}-q_{02} q_{13}+q_{01}q_{23},625 q_{01}^2+625 q_{02}^2+256 q_{03}^2+150 q_{01} q_{13}+9 q_{13}^2+150 q_{02} q_{23}+9 q_{23}^2 \rangle.
\end{aligned}
\]
We observe that $I_2$ is a component of $Ch_{X_{\rm Iso}}$. It is extraneous for us.
  More precisely, $X_{\rm Iso}$ contains two quadric curves on $S$, and $I_2$ is generated by the Pl\"ucker quadric and the Chow form of one of these curves. On the other hand, $I_1$  is the ideal of $C_S$.
  This  $(1,0)$-congruence is the $\alpha$-plane for $Q=(1:0:0:-3)$. 
   The points $P$ and $Q$ are the two foci of the ellipsoid~$S$.

   If we choose $P$ randomly, then $C_S=C_S'$.  Using the computation
   explained above, we find that the bidegree of the specular
   congruence $C_S$ is $(8,4)$. According to Josse and
   P\`ene~\cite{JP}, the focal locus $\mathcal{F}(C_S)$, which is the
   caustic by reflection of $S$, is a surface of degree $18$. 

Finally, we consider the catadioptric camera given by $S$ together with a general $(1,1)$-congruence 
(two-slit camera). The resulting specular congruence has bidegree $(12,6)$.
\hfill $\diamondsuit$
\end{example}

In closing, we wish to reiterate that the notion
of order used in this paper is always the algebraic one.
   The ``physical'' order of a catadioptric camera may be quite a
   bit lower,   due to orientation constraints,
   with some of the rays reflected inside the body of the mirror.

\section{Photographic Cameras}
\label{sec7}

The geometric cameras studied in the previous sections are maps
from $\PP^3$ to ${\rm Gr}(1,\PP^3)$. They do not require fixing image coordinates. A physical ``photographic'' camera, on the other hand, will always return measurements using image coordinates. Such a camera
is best modeled as a map $\PP^3 \dashrightarrow \PP^2$ or  $\PP^3 \dashrightarrow \PP^1 \times \PP^1$.
 In this final section, we examine general photographic cameras and their relationship with congruences
 and concurrent lines.

\subsection{Projections and Coordinates}
\label{subsec71}
We define a {\em photographic camera} to be a rational map $\PP^3 \dashrightarrow \PP^2$ or $\PP^3 \dashrightarrow \PP^1 \times \PP^1$ with the property that the fiber of a generic point is a line in $\PP^3$. This extends the traditional notion of a pinhole camera, which is a linear projection $\PP^3 \dashrightarrow \PP^2$ described by a $3 \times 4$ matrix. A photographic camera can be described explicitly  by a triplet 
$[f_0 : f_1 : f_2]$ of homogeneous polynomials in $\RR[x_0,x_1,x_2,x_3]$ of the same degree,
  or by two such pairs $([g_0 : g_1],[h_0:h_1])$. Of course, these polynomials cannot be general. 
  Algebraically, if $[f_0 : f_1 : f_2]$ is a photographic camera then the saturation of $\langle f_i u_j - f_j u_i \,:
  \, j \ne i\rangle$ with respect to $\langle f_0,f_1,f_2 \rangle$ in $\RR[x_0,x_1,x_2, x_3, u_0,u_1,u_2]$ 
  has two generators that are linear in the variables $x_0,x_1,x_2,x_3$.

\begin{example} \label{ex:fff}  \rm
For a photographic camera given by three
quadrics $f_0,f_1,f_2$, the base locus of the map given by $[f_0:f_1:f_2]$ must be a
curve of degree $3$. This is necessary and sufficient for the
requirement that the generic fiber is a line in $\PP^3$.
If the base locus is irreducible then this it is a twisted cubic curve in $\PP^3$.
Algebraically, this means that the three quadrics are the
$2 \times 2$-minors of a $2 \times 3$-matrix of linear forms in $x_0,x_1,x_2,x_3$.
\hfill $\diamondsuit$
\end{example}

 A photographic camera $M$ determines an injective rational map 
 $L_M: \PP^2 \dashrightarrow {\rm Gr}(1,\PP^3)$,
  or $L_M: \PP^1 \times \PP^1 \dashrightarrow {\rm Gr}(1,\PP^3)$, that associates 
  image points with their fiber. The closure of the image of $L_M$ is a congruence $C_M$ of order one.
  This is the congruence of all lines that are ``captured'' by the camera.
The rational camera (\ref{eq:cong_proj}) associated with $C_M$
satisfies  $C_M(x) = L_M(M(x))$ for generic points $x \in \PP^3$.
The base locus of $M$ contains the focal
locus of $C_M$. The photographic camera $M$ has  {\em class}
 $\beta$ if the congruence $C_M$ has bidegree $(1,\beta)$. In Example~\ref{ex:fff}, the class is $\beta = 3$, and $\mathcal{F}(C_M)$ is the
twisted cubic curve $V(f_0,f_1,f_2) \subset \PP^3$.

 Conversely, given any order one congruence $C$ and any birational map $G_C: C \dashrightarrow \PP^2$ or $G_C: C \dashrightarrow \PP^1 \times \PP^1$, we have that $x \mapsto G_C(C(x))$ is a photographic camera. In particular, we can use the classification of congruences in Section \ref{sec4} to construct photographic cameras.

\paragraph{Two-slit cameras.} A {\em linear two-slit camera} is a photographic camera $\PP^3 \dashrightarrow \PP^1 \times \PP^1$ of the form $x \mapsto (A x, B x)$, where $A$ and $B$ are $2 \times 4$-matrices
whose kernels are two skew lines in $\PP^3$.
 It is associated with the $(1,1)$-congruence $C$ of transversals to the two lines.  The
 formula for the rational map
  $\PP^1 \times \PP^1 \dashrightarrow C \subset {\rm Gr}(1,\PP^3)$ taking image points to their fibers~is 
   \begin{equation}\label{eq:parameterization_two_slit}
\begin{small}
\begin{bmatrix} u \\ v \end{bmatrix}\, \,\mapsto \,\,
\begin{bmatrix} A \\ B \end{bmatrix}^{-1} \begin{bmatrix} u \\ 0 \end{bmatrix}\, \vee\,
\begin{bmatrix} A \\ B \end{bmatrix}^{-1} \begin{bmatrix} 0 \\ v \end{bmatrix} 
\end{small}
\,\, = \,\,
u_0 v_0 D_{02} + u_0 v_1 D_{03} + u_1 v_0 D_{12} + u_1 v_1 D_{13} , 
\end{equation}
where $D_{ij}$ are column vectors of the $6 \times 6$-matrix 
$D = \wedge_2 \begin{small} \begin{bmatrix} A \\ B \end{bmatrix}^{-1} \end{small}$. Note that up to a common scale factor, $D_{i(k+2)}=(-1)^{i+k}(A_{\hat i} \wedge B_{\hat k})$, where $A_{\hat i}$ and $B_{\hat k}$ are rows of $A$ and $B$ and $(i,\hat i)$, $(k, \hat k)$ are pairs of distinct indices in $\{0,1\}$.
To obtain two-slit photographic cameras $\PP^3 \dashrightarrow \PP^2$, we can compose the linear two-slit camera with any birational map $\PP^1 \times \PP^1 \dashrightarrow \PP^2$.

\begin{example} \rm One photographic two-slit camera 
$\PP^3 \dashrightarrow \PP^2$ is $M(x) = (x_0x_3 : x_1 x_2 : x_1 x_3)$. This corresponds to the $(1,1)$-congruence of lines intersecting $L_1= V(x_2,x_3)$ and 
$L_2 =V(x_0,x_1)$. The map $M'(x)=(x_1 x_2 : x_0 x_3: x_0 x_2)$ is a different photographic camera 
that gives the same geometric camera. The two photographic cameras are related by $M'=M \circ \sigma $, where $\sigma$ is the Cremona transformation
$ \,\PP^2 \dashrightarrow \PP^2,\,w \mapsto (w_1 w_2 : w_0 w_2 : w_0 w_1)$. \hfill $\diamondsuit$
\end{example}

\paragraph{Cameras of higher class.} 
 Let $f, g, h$ be general binary forms of degree $\beta-1, \beta, \beta$ respectively,
 and let $A$ and $B$ be $2\times 4$-matrices as above (and $B$ has rows $B_1,B_2$). The~map  
\begin{equation}\label{eq:alg_camera_model}
\PP^3 \dashrightarrow \PP^1 \times \PP^1,\,
x \mapsto \left(Ax, \,\,\, \begin{pmatrix} 
g(Ax)-f(Ax) B_1 x \\ h(Ax)-f(Ax) B_2 x 
\end{pmatrix} \right)
\end{equation}
is a photographic camera of class $\beta$. 
Up to coordinate changes in $\PP^3$ we may assume
$A=\begin{pmatrix} 1 & 0 & 0 & 0 \\ 0 & 1 & 0 & 0 \end{pmatrix}$ and $B=\begin{pmatrix} 0 & 0 & 1 & 0 \\ 0 & 0 & 0 & 1 \end{pmatrix}$. 
The map (\ref{eq:alg_camera_model}) is undefined on the line $L=V(x_1,x_2)$ 
and on the parametric curve $X(s : t)$  as in \eqref{eq:Xpara}.
That curve has degree $\beta$ and it intersects $L$ in $\beta-1$ points.  The pre-image of a point  $(u,v)\in \PP^1 \times \PP^1$ is the line with Pl\"ucker coordinates
\[
\begin{bmatrix}
u_0 f(u_0,u_1) \\ u_1 f(u_0,u_1) \\ g(u_0,u_1) \\ h(u_0,u_1) \\
\end{bmatrix} \vee
\begin{bmatrix} 0 \\ 0 \\
 v_0 \\ v_1 
 \end{bmatrix} \,\,=\,\, \begin{bmatrix}
 0 \\  v_0 u_0 f(u_0, u_1) \\ v_1 u_0 f(u_0, u_1) \\ v_0 u_1 f(u_0, u_1) \\ v_1 u_0 f(u_0, u_1) \\ v_1 g(u_0,u_1) - v_0 h(u_0,u_1)
 \end{bmatrix}.
\]
This camera is a  $(1,\beta)$-congruence of type 3 as in Section \ref{sec4}.
 The points $u$ and $v$ are respectively the parameters for points on $X$ and $L$.  
 A photographic camera $\PP^3 \dashrightarrow \PP^2$ is
  obtained as in the two-slit case by composing \eqref{eq:alg_camera_model} with 
  a birational map $\PP^1 \times \PP^1 \dashrightarrow \PP^2$.
  

\begin{example} \rm  The map $M(x)=((x_0 : x_1),(x_0^2+x_1^2-x_0x_2 : x_0x_1 - x_0 x_3))$
is a photographic camera $\PP^3 \dashrightarrow \PP^1 \times \PP^1$ with $\beta = 2$. It corresponds to the congruence of lines intersecting $L= V(x_0,x_1)$ and $X(s:t)=(s^2 : st : s^2 + t^2 : st)$. A photographic camera  $\PP^3 \dashrightarrow \PP^2$ for the same congruence is $M'(x)=( x_0^3+x_0x_1^2-x_0^2x_2 : x_0^2 x_1-x_0^2 x_3 : x_0^2 x_1+ x_1^3-x_0x_1x_2)$. The base locus of $M'$ is the union of $X$ and $L$. 
\hfill $\diamondsuit$
\end{example}

\subsection{Multi-View Varieties and Fundamental Tensors}

Fix any collection of photographic cameras $M_1,\ldots, M_{n_1}, M'_{1}, \ldots , M'_{n_2}$ where $M_i: \PP^3 \dashrightarrow \PP^2$ and $M'_j: \PP^3 \dashrightarrow \PP^1 \times \PP^1$. 
The associated {\em multi-view variety} is the
closure of the image of
\begin{equation}\label{eq:joint_projection}
\PP^3 \dashrightarrow (\PP^2)^{n_1} \times (\PP^1 \times \PP^1)^{n_2}, 
\,\,\, x \mapsto \bigl( M_1(x),\ldots, M_{n_1}(x), M'_1(x), \ldots, M'_{n_2}(x) \bigr).
\end{equation}
This definition extends the usual notion of multi-view varieties
 in \cite{AST, THP}. The following result is the direct generalization of
Proposition \ref{prop:multiview} from pinhole cameras to arbitrary photographic cameras.
We write  $C_{M_i}$ and $C_{M'_j}$ for the congruences associated with $M_i$ and $M'_j$.

\begin{proposition} The multi-view variety for $M_1,\ldots, M_{n_1}, M'_{1}, \ldots , M'_{n_2}$ is birational
 to the multi-image variety $M(C_{M_1},\ldots, C_{M_{n_1}}, C_{M'_1},\ldots, C_{M'_{n_2}})$, defined in 
 Section \ref{sec5}, under the map
\begin{equation}\label{eq:photographic_multicamera}
L_{M_1} \times \cdots \times L_{M_{n_1}} \times L_{M'_1} \times \cdots \times L_{M'_{n_2}}: (\PP^2)^{n_1} \times (\PP^1 \times \PP^1)^{n_2} \dashrightarrow {\rm Gr(1,\PP^3)}^n,
\end{equation}
where $L_{M_i}, L_{M'_j}$
 are  parameterizations of the congruences $C_{M_i},C_{M'_j}$, as in Subsection \ref{subsec71}.
 \end{proposition}

 From Theorem \ref{thm:pairwisedisjoint} we deduce that, if the base
 loci of $M_1,\ldots, M_{n_1}, M'_{1}, \ldots , M'_{n_2}$ are pairwise
 disjoint, then the multi-view variety is birational to a slice of the
 concurrent lines variety $V_{n_1+n_2}$. Since the closure of the
 image of \eqref{eq:photographic_multicamera} is $C_{M_1}\times \cdots
 \times C_{M_{n_1}} \times C_{M'_1} \times \cdots \times
 C_{M'_{n_2}}$, we can obtain multi-view constraints in image
 coordinates by replacing the Pl\"ucker variables with the coordinates
 of $L_{M_i}(w)$ and $L_{M'_j}((u,v))$ in the multilinear polynomials
 defining $V_{n_1+n_2}$. We intend to investigate these constraints
 and their application to calibrating general cameras in another
 publication.  In the remainder of this paper we get started with a
 special case: we derive the epipolar constraint for two linear
 photographic cameras, pinhole or two-slit.

We first consider two pinhole cameras $\PP^3 \dashrightarrow \PP^2$, identified with two 
$3 \times 4$-matrices $A$ and $B$. The camera $A$ induces a parameterization 
$L_{A}:\PP^2 \rightarrow {\rm Gr}(1,\PP^3)$ of its $\alpha$-plane via
\begin{equation}\label{eq:parameterization_pinhole}
w \,\,\mapsto \,\, w_0 (A_{1} \wedge A_{2}) -  w_1 (A_{0} \wedge A_{2}) + w_2  (A_{1} \wedge A_{2}),
\end{equation}
where the $A_i$ are row vectors of $A$.
A similar expression holds for $L_B(u')$. 
Replacing Pl\"ucker coordinates in the incidence constraint ${\rm trace}(PQ^*)=0$ 
with the images of $w$ and $w'$, we obtain $w'^T F w=0$ where $F$ is the 
 {\em fundamental matrix} in multi-view geometry. Its entries~are 
  \[ f_{il}\,\,=\,\,  (-1)^{i+l} \cdot \det \begin{bmatrix} A_j^T & A_k^T & B_m^T & B_n^T
\end{bmatrix}, \]
where $(i,j,k)$ and $(l,m,n)$ are triplets of distinct indices. The (closure of the) set of all fundamental matrices 
is the cubic hypersurface in $\PP^8$ that is defined by the $3 \times 3$-determinant.

Now let $(A,B)$ and $(C,D)$ be pairs of $2 \times 4$ matrices describing two-slit cameras. The corresponding line congruences can be parametrized similarly to
(\ref{eq:parameterization_pinhole}), using
 \eqref{eq:parameterization_two_slit}. One deduces that
 a pair $\bigl((u,v),(u',v')\bigr) \in (\PP^1 \times \PP^1)^2 $ belongs to the multi-view variety if and only if 
 $\sum_{i,j,k,l=1}^2 f_{ijkl} \, u_i v_j u'_k v'_l=0$ where $F$ is the $2 \times 2 \times 2 \times 2$ 
 quadrifocal tensor, with entries
\[f_{ijkl}\,\,= \,\,(-1)^{i+j+k+l} \cdot \det \begin{bmatrix} A_{\hat i}^T &  B_{\hat j}^T &  C_{\hat k}^T  
&  D_{\hat l}^T  \end{bmatrix}. \]
The set of such tensors forms a $13$-dimensional variety in $\PP^{15}$.
According to \cite[Theorem 3]{LS}, this variety is defined by $718$ polynomials of degree $12$.
See \cite[Section 4.1]{Oed} for details and the connection to the more familiar
quadrifocal tensor of size $3 \times 3 \times 3 \times 3$.

Finally, let $A$  be a pinhole camera and $(B,C)$ is a two-slit camera.
By mixing the  two parametrizations used above,
   we obtain a $3 \times 2 \times 2$ epipolar tensor $F$ whose entries are
   \begin{equation}\label{eq:epi322}
f_{ijk}\,\,= \,\, (-1)^{i+j+k}  \cdot \det \begin{bmatrix} A_{l}^T & A_{m}^T &  B_{\hat j}^T &  C_{\hat k}^T \end{bmatrix}.
\end{equation}
Pairs of image points $\bigl(u,(u',v') \bigr) \in \PP^2 \times (\PP^1 \times \PP^1)$
that lie in the multi-view variety are characterized by $\sum_{i,j,k} f_{ijk} \, u_i u'_k v'_l=0$. 
The set of such tensors has codimension $1$ in $\PP^{11}$. 

\begin{proposition} The variety of $3 \times 2 \times 2$ tensors \eqref{eq:epi322} is the unique ${\rm SL}(3) \times {\rm SL}(2) \times
 {\rm SL}(2)$-invariant
 hypersurface of degree $6$ in $\,\PP^{11} = \PP( \CC^3 \times \CC^2 \times \CC^2)$.
 Its defining polynomial is
 \begin{small}
$$ 
\begin{matrix}
f_{111}^2 f_{212} f_{221} f_{322}^2
-f_{111}^2 f_{212} f_{222} f_{321} f_{322}
-f_{111}^2 f_{221} f_{222} f_{312} f_{322}
+f_{111}^2 f_{222}^2 f_{312} f_{321}  \\
-f_{111} f_{112} f_{211} f_{221} f_{322}^2
+f_{111} f_{112} f_{211} f_{222} f_{321} f_{322}
-f_{111} f_{112} f_{212} f_{221} f_{321} f_{322}
+f_{111} f_{112} f_{212} f_{222} f_{321}^2  \\
+f_{111} f_{112} f_{221}^2 f_{312} f_{322}
+f_{111} f_{112} f_{221} f_{222} f_{311} f_{322}
-f_{111} f_{112} f_{221} f_{222} f_{312} f_{321}
-f_{111} f_{112} f_{222}^2 f_{311} f_{321} \\
-f_{111} f_{121} f_{211} f_{212} f_{322}^2 
+f_{111} f_{121} f_{211} f_{222} f_{312} f_{322}
+f_{111} f_{121} f_{212}^2 f_{321} f_{322}
-f_{111} f_{121} f_{212} f_{221} f_{312} f_{322} \\
+f_{111} f_{121} f_{212} f_{222} f_{311} f_{322} 
-f_{111} f_{121} f_{212} f_{222} f_{312} f_{321}
+f_{111} f_{121} f_{221} f_{222} f_{312}^2
-f_{111} f_{121} f_{222}^2 f_{311} f_{312} \\
+f_{111} f_{122} f_{211} f_{212} f_{321} f_{322} 
+f_{111} f_{122} f_{211} f_{221} f_{312} f_{322}
-2 f_{111} f_{122} f_{211} f_{222} f_{312} f_{321}
-f_{111} f_{122} f_{212}^2 f_{321}^2 \\
-2 f_{111} f_{122} f_{212} f_{221} f_{311} f_{322} 
+2 f_{111} f_{122} f_{212} f_{221} f_{312} f_{321}
+f_{111} f_{122} f_{212} f_{222} f_{311} f_{321}
{-} f_{111} f_{122} f_{221}^2 f_{312}^2 \\
+f_{111} f_{122} f_{221} f_{222} f_{311} f_{312} 
+f_{112}^2 f_{211} f_{221} f_{321} f_{322}
-f_{112}^2 f_{211} f_{222} f_{321}^2
-f_{112}^2 f_{221}^2 f_{311} f_{322} \\
+f_{112}^2 f_{221} f_{222} f_{311} f_{321} 
+f_{112} f_{121} f_{211}^2 f_{322}^2
-f_{112} f_{121} f_{211} f_{212} f_{321} f_{322}
-f_{112} f_{121} f_{211} f_{221} f_{312} f_{322} \\
+f_{112} f_{121} f_{222}^2 f_{311}^2
+2 f_{112} f_{121} f_{211} f_{222} f_{312} f_{321}
+2 f_{112} f_{121} f_{212} f_{221} f_{311} f_{322}
-f_{112} f_{121} f_{212} f_{222} f_{311} f_{321} \\
-2 f_{112} f_{121} f_{211} f_{222} f_{311} f_{322} 
-f_{112} f_{121} f_{221} f_{222} f_{311} f_{312} 
{-}f_{112} f_{122} f_{211}^2 f_{321} f_{322}
{+}f_{112} f_{122} f_{211} f_{221} f_{311} f_{322}  \\
+f_{112} f_{122} f_{211} f_{212} f_{321}^2 
-f_{112} f_{122} f_{211} f_{221} f_{312} f_{321}
+f_{112} f_{122} f_{211} f_{222} f_{311} f_{321}
{-}f_{112} f_{122} f_{212} f_{221} f_{311} f_{321} \\
+f_{112} f_{122} f_{221}^2 f_{311} f_{312} 
{-}f_{112} f_{122} f_{221} f_{222} f_{311}^2
{+}f_{121}^2 f_{211} f_{212} f_{312} f_{322}
{-}f_{121}^2 f_{211} f_{222} f_{312}^2 
{-}f_{121}^2 f_{212}^2 f_{311} f_{322} \\
+f_{121}^2 f_{212} f_{222} f_{311} f_{312}
-f_{121} f_{122} f_{211}^2 f_{312} f_{322}
+f_{121} f_{122} f_{211} f_{212} f_{311} f_{322} 
-f_{121} f_{122} f_{211} f_{212} f_{312} f_{321} \\
+f_{121} f_{122} f_{211} f_{221} f_{312}^2
+f_{121} f_{122} f_{211} f_{222} f_{311} f_{312}
+f_{121} f_{122} f_{212}^2 f_{311} f_{321} 
-f_{121} f_{122} f_{212} f_{221} f_{311} f_{312}  \\
-f_{121} f_{122} f_{212} f_{222} f_{311}^2
{+}f_{122}^2 f_{211}^2 f_{312}f_{321}
{-}f_{122}^2 f_{211} f_{212} f_{311} f_{321} 
{-}f_{122}^2 f_{211} f_{221} f_{311} f_{312} 
{+}f_{122}^2 f_{212} f_{221} f_{311}^2
\end{matrix}
$$
\end{small}
\end{proposition}

\begin{proof}
The principal ideal of this hypersurface 
of $3 \times 2 \times 2$-tensors
can be computed by elimination  from the
prime ideal of the trifocal variety \cite{AO} in the space of
$3 \times 3 \times 3$-tensors.
\end{proof}

\medskip

\begin{small}
\noindent
{\bf Acknowledgements.}\smallskip \\
  This project started at the ``Algebraic Vision'' workshop held in May 2016 at the American 
  Institute of Mathematics (AIM) in  San Jose.  We are grateful to the organizers,  Sameer Agarwal, Max Lieblich and Rekha Thomas, for bringing us together. We also thank John
Canny, Xavier Goaoc, Martial Hebert, Joe Kileel, Kathl\'en Kohn, Luke Oeding
and Fran\c coise P\`ene for helpful comments and discussions.
 Bernd Sturmfels was supported in part by the
  US National Science Foundation  (DMS-1419018)
  and the Einstein Foundation Berlin.
Jean Ponce and Mathew Trager were supported in part by the ERC
advanced grant VideoWorld and the Institut Universitaire de France.
\end{small}

\bigskip

\noindent
\footnotesize {\bf Authors' addresses:}

\smallskip

\noindent  Jean Ponce, 
\'Ecole Normale Sup\'erieure/PSL Research Univ.~and INRIA Paris, France,
{\tt Jean.Ponce@ens.fr}

\smallskip

\noindent Bernd Sturmfels,
University of California, Berkeley, USA,
{\tt bernd@berkeley.edu}

\smallskip

\noindent Mathew Trager,
INRIA Paris, France,
{\tt matthew.trager@inria.fr}


\medskip

\begin{thebibliography}{10}
\setlength{\itemsep}{-1mm}  

\bibitem{ABT} E.~Arrondo, M.~Bertolini and C.~Turrini:
{\em A focus on focal surfaces}, Asian Journal of Mathematics~{\bf 5} (2001) 535--560.

\bibitem{AO} C.~Aholt and L.~Oeding: {\em The ideal of the trifocal variety}, Mathematics of Computation,
{\bf 83} (2014) 2553--2574.

\bibitem{AST} C.~Aholt, B.~Sturmfels and R.~Thomas: {\em A Hilbert scheme in 
computer vision}, Canadian Journal of Mathematics {\bf 65} (2013) 961--988. 

\bibitem{BN} S.~Baker and S.K.~Nayar: {\em A theory of single-viewpoint catadioptric image formation}, International Journal of Computer Vision {\bf 35} (1999) 175--196.

\bibitem{BGP} G.~Batog, X.~Goaoc and J.~Ponce: {\em Admissible linear
    map models of linear cameras}, 2010 IEEE Conference on Computer
  Vision and Pattern Recognition (CVPR), 2010.

 \bibitem{BG} V.~Beni\'c and S.~Gorjanc: {\em (1, n) Congruences}, KoG: 
 Scientific and Professional Journal
 of the Croatian Society for Geometry and Graphics   {\bf 10} (2007) 5--12.

\bibitem{CLO} D.~Cox, J.~Little and D.~O'Shea: {\em Ideals, Varieties and Algorithms}, Undergraduate Texts in Mathematics, Springer, New York, 2007.

\bibitem{CT} F.~Catanese and C.~Trifogli: {\em Focal loci of algebraic varieties I.}, Communications in Algebra  {\bf 28} (2000) 6017--6057.

\bibitem{DS} J.~Dalbec and B.~Sturmfels: {\em
Introduction to Chow forms}, in "Invariant Methods in Discrete and Computational Geometry" (N. White, ed.), Proceedings Curacao (June 1994), Kluwer Academic Publishers, 1995, pp. 37--58.

\bibitem{DeP} P.~De Poi: {\em Congruences of lines with one-dimensional focal locus},
 Portugaliae Mathematica {\bf 61} (2004) 329--338.

\bibitem{DM} P.~De Poi and E.~Mezzetti: 
{\em On a class of first order congruences of lines},
Bull.~Belg.~Math.~Soc.~Simon Stevin {\bf 16} (2009) 805--821.

\bibitem{DH} J.~Draisma, E.~Horobet, G.~Ottaviani, B.~Sturmfels and R.~Thomas: {\em The Euclidean distance degree of an algebraic variety}, Found.~Comput.~Math.~{\bf 16} (2016) 99--149.

\bibitem{EK} L.~Escobar and A.~Knutson:
The multidegree of the multi-image variety,
in {\em Combinatorial Algebraic Geometry}
(eds.~G.G.~Smith and B.~Sturmfels), to appear.


\bibitem{M2} D.~Grayson and M.~Stillman:                                      
{\em Macaulay2, a software system for research in algebraic geometry}, available at 
{\tt www.math.uiuc.edu/Macaulay2/}. 

\bibitem{GH}
R.~Gupta and R.~Hartley: {\em Linear pushbroom cameras},
IEEE Transactions on Pattern Analysis and Machine Intelligence
{\bf 19} (1997) 963--975.

\bibitem{HZ} R.~Hartley and A.~Zisserman: {\em Multiple View
Geometry in Computer Vision}, Cambridge University Press, 2000.        

\bibitem{Jes} C.M.~Jessop: {\em A Treatise on the Line Complex},
Cambridge University Press, 1903, (American Mathematical Society, 2001).

\bibitem{JP} A.~Josse and F.~P\`ene: {\em On caustics by reflection of algebraic surfaces},
Advances in Geometry {\bf 16} (2016) 437--464.
                          
\bibitem{Kum} E.~Kummer: {\em \"Uber die algebraischen Strahlensysteme,
insbesondere \"uber die der ersten und zweiten Ordnung},
Abh.~K.~Preuss.~Akad.~Wiss.~Berlin (1866) 1--120.

\bibitem{Li} B.~Li: {\em Images of rational maps of projective spaces},
{\tt arXiv:1310.8453}.

\bibitem{LS} S.~Lin and B.~Sturmfels:
{\em Polynomial relations among principal minors of a $4 \times 4$ matrix},
  Journal of Algebra {\bf 322} (2009) 4121--4131. 
 
\bibitem{MS} E.~Miller and B.~Sturmfels: {\em Combinatorial Commutative Algebra},
Graduate Texts in Mathematics {\bf 227}, Springer, New York, 2004.

\bibitem{Oed} L.~Oeding: {\em The quadrifocal variety},
Linear Algebra and its Applications {\bf 512} (2017) 306--330. 

\bibitem{Paj} T.~Pajdla: {\em Stereo with oblique cameras},  
International Journal of Computer Vision {\bf 47} (2002) 161--170.

\bibitem{Pon} J.~Ponce: {\em What is a camera?}, in
IEEE Conference on Computer Vision and Pattern Recognition (CVPR), 2009.


\bibitem{HKS} H.Y.~Shum, A.~Kalai and S.M.~Seitz: {\em Omnivergent stereo},
 Proceedings of the IEEE International Conference on Computer Vision, 1999.

\bibitem{SRT} P.~Sturm, S.~Ramalingam, J.P.~Tardif, S.~Gasparini and J.~Barreto: {\em Camera models and fundamental concepts used in geometric computer vision}, Foundations and Trends in Computer Graphics and Vision {\bf 6} (2011) 1--183.

\bibitem{THP}
M.~Trager, M.~Hebert and J.~Ponce: {\em The joint image handbook}, Proceedings of the IEEE International Conference on Computer Vision, 2015.

\bibitem{TPH}
M.~Trager, J.~Ponce and M.~Hebert: {\em Trinocular geometry revisited}, 
International Journal on Computer Vision, 2016, on-line first.

\end{thebibliography}
\end{document}